\documentclass[a4paper,12pt]{article}

\usepackage[left=2cm,right=2cm, top=2cm,bottom=2cm,bindingoffset=0cm]{geometry}

\usepackage{verbatim}
\usepackage{amsmath}
\usepackage{amsthm}
\usepackage{amssymb}
\usepackage{delarray}
\usepackage{cite}

\newcommand{\al}{\alpha}

\newcommand{\ga}{\gamma}
\newcommand{\de}{\delta}

\newcommand{\om}{\omega}
\newcommand{\ee}{\varepsilon}
\newcommand{\eps}{\varepsilon}
\newcommand{\vv}{\varphi}

\theoremstyle{plain}

\newtheorem{thm}{Theorem}
\newtheorem{lem}{Lemma}

\theoremstyle{remark}

\newtheorem{remark}{Remark}

\DeclareMathOperator*{\Res}{Res}

 \allowdisplaybreaks

\begin{document}

\begin{center}
{\large\bf
Recovery of the matrix quadratic differential pencil

from the spectral data
}
\\[0.2cm]
{\bf Natalia Bondarenko} \\[0.2cm]
\end{center}

\vspace{0.5cm}

{\bf Abstract.} We consider a pencil of matrix Sturm-Liouville operators on a finite interval.
We study properties of its spectral characteristics and inverse problems that
consist in recovering of the pencil by the spectral data:
eigenvalues and so--called weight matrices. This inverse problem is reduced to a linear equation
in a Banach space by the method of spectral mappings. Constructive algorithm for the solution
of the inverse problem is provided. 

\medskip


\medskip


\vspace{1cm}

{\large \bf 1. Introduction} \\

In this paper, we consider the boundary value problem $L = L(\ell, U, V)$ for the equation
\begin{equation} \label{eqv}
    \ell Y := Y'' + (\rho^2 I + 2 i \rho Q_1(x) + Q_0(x)) Y = 0, \quad x \in (0, \pi),
\end{equation}
with the boundary conditions
\begin{equation} \label{BC}
    \begin{array}{l}
        U(Y) := Y'(0) + (i \rho h_1 + h_0) Y(0) = 0, \\
        V(Y) := Y'(\pi) + (i \rho H_1 + H_0) Y(\pi) = 0.
    \end{array}
\end{equation}
Here $Y(x) = [y_k(x)]_{k = \overline{1, m}}$ is a column vector, $\rho$ is the spectral parameter,
$I$ is the $m \times m$ unit matrix,
$Q_s(x) = [Q_{s, jk}(x)]_{j,k = \overline{1, m}}$ are $m \times m$ matrices with entries
$Q_{s, jk}(x) \in W_1^s[0, \pi]$, $s = 0, 1$
$h_s = [h_{s, j k}]_{j, k = \overline{1, m}}$,
$H_s = [H_{s, j k}]_{j, k = \overline{1, m}}$,
where $h_{s, j k}$, $H_{s, j k}$ are complex numbers.

Differential equations with nonlinear dependence on the spectral parameter,
or with so-called ``energy--dependent'' coefficients, frequently appear in
mathematics and applications (see \cite{Kel71, KS83, Shkal83, Markus86, Yur97} and references therein).
In this work, we study the inverse problem of spectral analysis for 
the matrix pencil $L$, which consists in recovering the coefficients in \eqref{eqv} and \eqref{BC}
from the spectral characteristics of $L$.
In the scalar case ($m = 1$), such problems were studied in papers \cite{GG81, BY06, BY12, HP12}.

If $Q_1(x) \equiv 0$, then $L$ becomes a matrix Sturm-Liouville operator.
The inverse problems theory for these operators 
was developed in works
\cite{AM60, Yur06, CK09, MT10, Bond11, Bond12}. 
We have also studied inverse problems for quadratic pencils of matrix Sturm-Liouville operators
in \cite{BF13, Bond-uniq}. The case of the pencil on a half-line \cite{BF13} appeared to be easier
for investigation since this pencil has only a finite number of eigenvalues. 
In \cite{Bond-uniq},
the case of a finite interval was considered and
the uniqueness theorem for the inverse problem was proved.
In the present paper, we continue this research.
Our goal is to provide a constructive procedure for recovering of the pencil $L$
from its spectral data.  

\smallskip

We assume that the coefficients of the pencil satisfy the conditions:

\begin{enumerate}
\item[(I)] $\det (I \pm h_1) \ne 0$, $\det (I \pm H_1) \ne 0$;
\item[(II)] $Q^{\dagger}_1(x) = -Q_1(x)$,
$Q^{\dagger}_0(x) = Q_0(x)$ for a.e. $x \in [0, \pi]$, $h^{\dagger}_1 = -h_1$, $H^{\dagger}_1 = -H_1$,
$h^{\dagger}_0 = h_0$, $H^{\dagger}_0 = H_0$ (the symbol $\dagger$ denotes the conjugate transpose).
\end{enumerate}

Condition (I) excludes problems of Regge type (see \cite{Yur84}) from the consideration,
as they require a separate investigation.

Let $\Phi(x, \rho) = [\Phi_{jk}(x, \rho)]_{j, k = \overline{1, m}}$
be the matrix solution of equation \eqref{eqv} under the conditions
$U(\Phi) = I$, $V(\Phi) = 0$. We call $\Phi(x, \rho)$ the
\textit{Weyl solution} for $L$. Put $M(\rho) := \Phi(0, \rho)$.
The matrix $M(\rho) = [M_{jk}(\rho)]_{j,k = \overline{1, m}}$ is called
the \textit{Weyl matrix} for $L$. The notion of the Weyl matrix
is a generalization of the notion of the Weyl function ($m$-function)
for the scalar case (see \cite{Mar77, FY01}) and the notion of the
Weyl matrix for the matrix Sturm--Liouville operator (see \cite{Bond11, Bond12}).

We will prove the following properties of spectral characteristics of $L$.

\begin{thm} \label{thm:SD}
The pencil $L$ has a countable set of eigenvalues;

(i) they can be numbered (counting with their multiplicities) to satisfy the following asymptotics as $|n| \to \infty$:
\begin{equation} \label{asymptrho}
\rho_{nq}=n+\om_q+O\left( n^{-1}\right) ,\quad n \in w:=\{\pm 0,\pm 1,\pm 2,\dots\} ,\quad q=\overline{1,m},\quad \om_q\in[0,1);
\end{equation}

(ii) for some sufficiently large $N$, all the eigenvalues $\rho_{nq}$, $|n| \ge N$, are real.

The poles of the Weyl matrix $M(\rho)$ coincide with the eigenvalues of $L$;
for some sufficiently large $N$,

(iii) the poles at $\rho_{nq}$, $|n| \ge N$, are simple;

(iv) the weight matrices $\al_{nq} := \Res_{\rho = \rho_{nq}} M(\rho)$ satisfy the asymptotics
\begin{equation} \label{asymptalpha}
\sum_{s\in J_q}\frac{\al_{n s}}{m_{n s}}=\frac{1}{\pi n}(I-h_1)^{-1}W^\dagger I_q W(I+h_1)^{-1}+O\left(n^{-2}\right),\quad |n|\rightarrow\infty,\quad q=\overline{1,m},
\end{equation}
and $\mbox{rank}\, \al_{nq} = m_{nq}.$ Here $m_{nq}$ are the multiplicities of the eigenvalues $\rho_{nq}$, $J_q = \{s \colon \om_q = \om_s \}$
and $W, I_q$ are some $m \times m$ matrices
(see notation before Lemma~\ref{lem:asymptalpha} for details).

\end{thm}

Some of these properties were proved in \cite{CW12} but 
for the matrix differential pencil with other boundary conditions,
independent of the spectral parameter. 
Moreover, we need here more precise asymptotics for the recovery of the pencil. 

The properties, established in Theorem~\ref{thm:SD},
are analogous to the ones for self-adjoint matrix Sturm-Liouville operators, 
but only for sufficiently large eigenvalues. The pencil
can have a finite number of nonreal eigenvalues and multiple poles of the Weyl matrix, even when the coefficients satisfy (II). 
This is possible even in the scalar case ($m = 1$). 

\medskip

{\bf Example.} Let $m = 1$, $h_1 = H_1 = h_0 = H_0 = 0$, $Q_1(x) = - i p$, $Q_0(x) = q$, 
where $p$ and $q$ are real constants. Clearly, conditions (I) and (II) are fulfilled.
So we have a scalar boundary value problem
$$
	y'' + (\rho^2 + 2 \rho p + q) y = 0, \quad y'(0) = y'(\pi) = 0.
$$
Its eigenvalues are the roots of the equations $\rho^2 + 2 \rho p + q = n^2$, $n = 0, 1, 2, \dots$. 
Obviously, these equations can have nonreal roots. Now suppose $p = q = 0$. Then one can easily calculate the Weyl function
$M(\rho) = \dfrac{\cos \rho \pi}{\rho \sin \rho \pi}$. Clearly, the eigenvalue $\rho = 0$ is a double pole of $M(\rho)$.
%These examples show that even under restrictions (I) and (II), there can exist nonreal eigenvalues and multiple poles of the Weyl function
%(matrix).

\medskip 

For convenience, renumerate the eigenvalues 
$\{ \rho_{nq}\}_{|n| < N}$ by $\{ \rho_n\}_{n = \overline{1, K}}$.
Denote by $m_n$ the multiplicity of $\rho_n$ as a pole of $M(\rho)$
(then we assume $\rho_n = \rho_{n + 1} = \dots = \rho_{n + m_n - 1}$).
Let $S = \{ n \in [1, K] \colon n = 1 \, \text{or} \, \rho_{n - 1} \ne \rho_n\}$.
It is easy to show that the Weyl matrix admits the following representation:
\begin{equation} \label{reprM}
	M(\rho) = \sum_{n \in S} \sum_{\nu = 0}^{m_n - 1} \frac{\al_{n + \nu}}{(\rho - \rho_n)^{\nu + 1}} +
	\sum_{|n| \ge N} \sum_{q = 1}^m \frac{m_{nq}^{-1} \al_{nq}}{(\rho - \rho_{nq})},
\end{equation}
where $\al_n$, $n = \overline{1, K}$, are some $m \times m$ matrices.
\footnote{In case of multiple eigenvalues, the same weight matrices $\al_{nq}$ will be counted multiple times. 
To avoid this problem, we divide the weight matrices in \eqref{asymptalpha} and \eqref{reprM}
by the multiplicities $m_{nq}$.}

We call the data $\mbox{SD} := \{ \rho_n, \al_n\}_{n = \overline{1, K}} \cup \{ \rho_{nq}, \al_{nq}\}_{|n| \ge N, q = \overline{1, m}}$
{\it the spectral data} of the pencil $L$. This notion generalizes the notion of spectral data for scalar pencils (see \cite{BY12}).
According to \eqref{reprM}, the spectral data determine the Weyl matrix uniquely.
Therefore, the inverse problem by the Weyl matrix, studied in \cite{Bond-uniq}, is equivalent to the following one.

\medskip

{\bf Inverse Problem 1.} {Given spectral data $\mbox{SD}$,
construct the coefficients of the pencil $L$.}

\medskip

Thus, in this paper we develop a constructive solution of Inverse Problem 1. 
The analogous results for the scalar case were obtained in \cite{BY06, BY12} for pencils with arbitrary complex-valued potentials.
However, in our paper we put an additional restriction (II) on the coefficients of the pencil, in order to  
guarantee the simplicity of the poles of the Weyl matrix at large eigenvalues and to work with asymptotics of the corresponding residues.
The scalar case is easier from this point of view, because in the scalar case asymptotic formula \eqref{asymptrho} yields the separation
of the eiganvalues for large $n$, so the poles of the Weyl function are automatically simple. 
In the matrix case, there can be infinitely many groups of multiple eigenvalues, and the general situation 
is very difficult for investigation.

The paper is organized as follows. In \textit{Section 2}, we study the properties of the spectral
characteristics of $L$, crucial for our method of recovery,
and prove Theorem~\ref{thm:SD}.
In \textit{Section 3}, following the ideas of the method of spectral mappings (see \cite{FY01, Yur02})
we choose a model pencil $\tilde L$ and then by integration in the complex plane of the spectral parameter,
reduce Inverse Problem 1 to a linear equation (see eq.\eqref{cont3} and Lemma~\ref{lem:contour}).
In \textit{Section 4}, this equation is transformed into the so-called \textit{main equation} \eqref{main} in a specially constructed Banach space.
We investigate the solvability of this equation
(Theorems~\ref{thm:homo} and \ref{thm:main}).
As opposed to the similar equation for the Sturm-Liouville operator, the main equation
for the pencil is uniquely solvable not for all $x \in [0, \pi]$. 
This difficulty appears even in the scalar case (see \cite{BY06, BY12}).
In \textit{Section 5}, we provide an algorithm for the solution of Inverse Problem~1 ``step by step'', choosing
a new model pencil at each step.
We prove that our algorithm finds coefficients of $L$ on the whole interval $[0, \pi]$
after a finite number of steps.
The algorithm, based on analogous ideas, was developed earlier
for the matrix pencil on the half-line \cite{BF13}.
For simplicity, we work under the assumption $N = 0$, i.e. all the eigenvalues and weight matrices 
satisfy properties (ii), (iii) of Theorem~\ref{thm:SD}.
In \textit{Section 6}, we show how to work with multiple poles of the Weyl matrix. 
We provide necessary modifications of the main equation and, finally, arrive at the solution
of the inverse problem in the general case.

\medskip

{\bf Notation.} Below along with $L$ we consider pencils $\tilde L = L(\tilde \ell, \tilde U, \tilde V)$ and
$L^* = L^*(\ell^*, U^*, V^*)$:
\begin{gather*}
\tilde \ell Y := Y'' + (\rho^2 I + 2 i \rho \tilde Q_1(x) + \tilde Q_0(x)) Y = 0, \quad x \in (0, \pi), \\
\begin{array}{l}
\tilde U(Y) := Y'(0) + (i \rho \tilde h_1 + \tilde h_0) Y(0) = 0, \\
\tilde V(Y) := Y'(\pi) + (i \rho \tilde H_1 + \tilde H_0) Y(\pi) = 0;
\end{array}\\
\ell^* Z := Z'' + Z (\rho^2 I + 2 i \rho Q_1(x) + Q_0(x)) = 0, \quad x \in (0, \pi), \\
\begin{array}{l}
U^*(Z) := Z'(0) + Z(0) (i \rho h_1 + h_0) = 0, \\
V^*(Z) := Z'(\pi) + Z(\pi)(i \rho H_1 + H_0) = 0,
\end{array}
\end{gather*}
where $Z$ is a row vector, and $\tilde L^*$, formed in an obvious way.
We agree that if a symbol $\gamma$ denotes an object related to $L$ then $\tilde \gamma$,
$\gamma^*$ and $\tilde \gamma^*$ denote
the corresponding objects related to $\tilde L$, $L^*$, $\tilde L^*$, respectively.
We mention once more that the conjugate transpose is denoted by $\dagger$. If $Y(x, \rho)$ is a solution of \eqref{eqv}, then
$Y^*(x, \rho) := Y^{\dagger}(x, \bar \rho)$ is a solution of $\ell^* Y^* = 0$.

We consider the space of complex $m$-vectors $\mathbb{C}^m$ with the norm
$$
	\| Y \| = \max_{1 \le j \le m} |y_j|, \quad Y = [y_j]_{j = \overline{1, m}},
$$
the space of complex $m \times m$ matrices $\mathbb{C}^{m \times m}$ with the corresponding induced norm
$$
\| A \| = \max_{1 \le j \le m} \sum_{k = 1}^m |a_{jk} |, \quad A = [a_{jk}]_{j, k = \overline{1, m}}.
$$
and the Hilbert space $L_2((0, \pi), \mathbb{C}^m)$ of $m$-vectors with entries from $L_2(0, \pi)$ with the
following scalar product and the corresponding norm:
$$
 	(Y, Z) = \int_0^{\pi} Y^{\dagger}(x) Z(x) \, dx = \int_0^{\pi} \sum_{j = 1}^m \bar{y}_j(x) z_j(x) \, dx.
$$
$\langle Y, Z \rangle$ denotes the matrix Wronskian $Y' Z - Y Z'$.

In estimates and asymptotics, we use the same symbol $C$ for different constants
independent of $x$, $\rho$, etc, and $\tau := \mbox{Im} \, \rho$.

\bigskip

{\large \bf 2. Properties of spectral characteristics}

\bigskip

Let $\vv(x, \rho) = [\vv_{jk}(x, \rho)]_{j, k = \overline{1, m}}$ and
$S(x, \rho) = [S_{jk}(x, \rho)]_{j, k = \overline{1, m}}$
be the matrix solutions of equation \eqref{eqv}
under the initial conditions
$$
 	S(0, \rho) = U(\vv) = 0, \quad S'(0, \rho) = \vv(x, \rho) = I.
$$

Introduce the matrix functions $P_{+}(x)$ and $P_{-}(x)$
as the solutions of the Cauchy problems
\begin{equation} \label{cauchyP}
	P_{\pm}'(x) = \pm Q_1(x) P_{\pm}(x), \quad P_{\pm}(0) = I.
\end{equation}
When condition (II) holds, they satisfy the relations
\begin{equation} \label{propP}
   P^{\dagger}_{+}(x) P_{+}(x) = P^{\dagger}_{-}(x) P_{-}(x) = I, \quad x \in [0, \pi].
\end{equation}
It was proved in \cite{Bond-uniq}, that for $x \in[0, \pi]$, $\nu = 0, 1$, as $|\rho| \to \infty$:
\begin{equation} \label{asymptS}
S^{(\nu)}(x, \rho)  = \frac{(i \rho)^{\nu - 1}}{2} \exp(i \rho x) P_{-}(x)  + \frac{(-i \rho)^{\nu - 1}}{2}
\exp(- i \rho x) P_{+}(x) + O(\rho^{\nu-2} \exp(|\tau|x)),
\end{equation}
\begin{multline} \label{asymptphi}
 	\vv^{(\nu)}(x, \rho) = \frac{(i \rho)^{\nu}}{2} \exp(i \rho x) P_{-}(x) (I - h_1) +
 	\frac{(- i \rho)^{\nu}}{2} \exp(- i \rho x) P_{+}(x) (I + h_1) + \\ + O(\rho^{\nu -1} \exp(|\tau|x)).
\end{multline}

\begin{lem} \label{lem:asymptrho}
The pencil $L$ has a countable set of eigenvalues. They can be numbered (counting with their multiplicities) in such a way 
that the eigenvalues have the following asymptotics as $|n| \to \infty$:
\begin{equation*}
\rho_{nq}=n+\om_q+O\left( n^{-1}\right) ,\quad n \in w:=\{\pm 0,\pm 1,\pm 2,\dots\} ,\quad q=\overline{1,m},\quad \om_q\in[0,1).
\end{equation*}
\end{lem}

\begin{proof}
It is clear that the eigenvalues of the pencil $L$ coincide with the zeros of the characteristic function
$\det V(\vv)$. This function is entire and has an at most countable set of zeros.
By \eqref{asymptphi},
\begin{multline} \label{asymptVphi}
V(\vv)=i\rho \Bigl\{\frac{1}{2}\exp(i \rho\pi) ( I+H_1) P_-(\pi)( I-h_1) \\
-\frac{1}{2}\exp(-i \rho\pi)( I-H_1)P_+(\pi)( I+h_1) \Bigr\}
+O\Bigl( \exp\bigl(|\tau|\pi\bigr)\Bigr),
\end{multline}
\begin{multline*}
\det V(\vv)=\left( \frac{i \rho}{2}\right)^m \det P_-(\pi)\det\left( I+H_1\right)\det\bigl(\exp(i \rho\pi)I-\exp(-i \rho \pi) A\bigr) \\
+O\Bigl(|\rho|^{m-1}\exp\bigl(m |\tau|\pi\bigr)\Bigr),\quad |\rho|\rightarrow\infty,
\end{multline*}
where
\begin{equation} \label{defA}
A=P_-^{-1}(\pi)\left( I+H_1\right) ^{-1}\left( I-H_1\right) P_+(\pi)\left( I+h_1\right)\left( I-h_1\right)^{-1}.
\end{equation}

It follows from \eqref{propP} and (II), that $\det P_-(\pi)$, $\det\left( I+H_1\right)$ and $\det\left( I-H_1\right)$ are nonzero. Therefore the eigenvalues coincide with the zeros of the analytic function of the form
\begin{equation}\label{initDR}
\Delta(\rho):=\det\Bigl(\exp(i \rho\pi)I-\exp(-i \rho \pi)A+O\bigl(|\rho|^{-1}\exp(|\tau|\pi)\bigr)\Bigr),\quad |\rho|\rightarrow\infty.
\end{equation}

The function $\Delta_0(\rho):=\det\bigl(\exp(i \rho\pi)I-\exp(- i \rho \pi)A\bigr)$ has zeros
\begin{equation*}
\rho_{nq}^0=n+\om_q,\quad n \in \mathbb{Z},\quad q=\overline{1,m},
\end{equation*}
where $w_q=\frac{1}{2\pi i}\ln\mu_q$, $\mu_q$ are the eigenvalues of the matrix $A$.

Since by \eqref{propP} and (II) the matrices $P_-(\pi)$, $P_+(\pi)$,
$\left( I+H_1\right)^{-1}\left( I-H_1\right)$ and $\left( I+h_1\right)\left( I-h_1\right)^{-1}$
are unitary, the matrix $A$ is also unitary. Its eigenvalues are simple and have the form
$\mu_q=e^{i\vv_q}$, where $\vv_q$ can be chosen in the interval $[0,2\pi)$.
Consequently, $\om_q=\frac{\vv_q}{2\pi}\in[0,1)$.

Consider the circles $\ga_{nq}$ with the centers $\rho_{nq}^0$ and radii equal to $\delta$, such that
\begin{equation*}
0<\delta\leq\frac{1}{2}\inf\left|\rho_{nq}^0-\rho_{mp}^0\right|,
\end{equation*}
where the infimum is taken over all $n,k\in \mathbb{Z}$, $1\leq q,p\leq m$, such that $\rho_{nq}^0\not=\rho_{mp}^0$.

Note that $\left|\Delta_0(\rho)\right|$ is bounded from below for
$\rho\in\ga_{nq},\ n\in \mathbb{Z},\ q=\overline{1,m}$. Furthermore,
$\left|\Delta(\rho)-\Delta_0(\rho)\right|=O(n^{-1})$ yields
$\left|\Delta(\rho)-\Delta_0(\rho)\right|\leq\left|\Delta_0(\rho)\right|$,
for $\rho\in\ga_{nq},\ n\in \mathbb{Z},\ q=\overline{1,m}$ for sufficiently large $|n|$.
Therefore we can apply Rouche's theorem to the analytic functions
$\Delta_0(\rho)$ and $\Delta(\rho)-\Delta_0(\rho)$ in the circle $\ga_{nq}$ and
conclude that $\Delta(\rho)$ has the same number of zeros as $\Delta_0(\rho)$ in this circle.
As $\delta$ tends to zero, we get
\begin{equation}\label{initR}
\rho_{nq}=\rho_{nq}^0+\om_q+\ee_{nq},\quad \ee_{nq}=o(1),\quad |n|\to\infty,\quad q=\overline{1,m}.
\end{equation}

Since the matrix $A$ is unitary, it can be represented in the form
\begin{equation}\label{initAWM}
A=W^\dagger \mathcal{M} W,\quad W^\dagger=W^{-1},\quad \mathcal{M} = \mbox{diag} \, \{\mu_1,\mu_2,\dots,\mu_m\}.
\end{equation}
Substituting  \eqref{initR} and \eqref{initAWM} into \eqref{initDR}, we get
\begin{equation*}
\det\Bigl( \mu_q\exp(2 \pi i\ee_{nq}) I - \mathcal{M}+O\left(n^{-1}\right)\Bigr)=0.
\end{equation*}
Since $\mathcal{M}$ is diagonal, we conclude that
\begin{equation} \label{esteps}
	\ee_{nq}=O\left(n^{-1}\right), \quad |n|\rightarrow\infty.
\end{equation}

Consider the contour
$\Gamma_R=\{\rho\colon|\rho|=R\}$ with a sufficiently large $R$,
that does not contain points $\rho_{nq}^0$.
Apply Rouche's theorem to it to the functions
$(i\rho)^m\Delta(\rho)$ and $(i\rho)^m\Delta_0(\rho)$.
As a result, $\det V(\vv)$ has the same number of zeros as
$(i\rho)^m\Delta_0(\rho)$ inside this contour, and together with \eqref{initR} and \eqref{esteps}
this yields the assertion of the Lemma.
\end{proof}

\begin{lem}\label{lem:symeq}
The following relation holds
\begin{equation} \label{symeq}
\int\limits_0^{\pi }\vv^\dagger(x,\rho_{nq})\vv(x,\rho_{nq})\,dx=\frac{\pi}{2}\left(I+h_1^\dagger h_1\right)+O\left(n^{-1}\right),\quad |n|\rightarrow\infty,\quad q=\overline{1,m}.
\end{equation}
Note that $\left(I+h_1^\dagger h_1\right)\ge 0$.
\end{lem}

\begin{proof}
We will use the notation $[I] = I + O(|\rho|^{-1})$. According to \eqref{asymptphi},
\begin{multline*}
\vv^\dagger(x,\rho)\vv(x,\rho)=\frac{1}{4}\biggl( \exp (-i\rho x)\left(I-h_1^\dagger\right)P_-^\dagger(x)[I]+\exp (i\rho x)\left(I+h_1^\dagger\right)P_+^\dagger(x)[I]\biggr)\\
\cdot\bigl(\exp (i\rho x)P_-(x)\left(I-h_1\right)[I]+\exp (-i\rho x)P_+(x)\left(I+h_1\right)[I] \bigr).
\end{multline*}
Using \eqref{propP}, we derive
\begin{multline} \label{smeq1}
\vv^\dagger(x,\rho)\vv(x,\rho)
=\frac{1}{2}\left(I+h_1^\dagger h_1\right)[I]+\frac{1}{4}\exp (2i\rho x)\left(I-h_1\right)P_+^\dagger(x)P_-(x)\left(I-h_1\right)[I] \\
+\frac{1}{4}\exp (-2i\rho x)\left(I+h_1\right)P_-^\dagger(x)P_+(x)\left(I+h_1\right)[I].
\end{multline}
Integrating by parts, we obtain
\begin{multline*}
\int\limits_0^{\pi }\exp (2i\rho\pi)P_+^\dagger(x)P_-(x)\,dx=\left.\frac{1}{2i\rho}\exp (2i\rho x)P_+^\dagger(x)P_-(x)\right|_0^\pi\\
-\frac{1}{2i\rho}\int\limits_0^{\pi }\exp (2i\rho x)P_+^\dagger(x)Q_1(x)P_-(x)\,dx=O\left(n^{-1}\right),\quad\rho=\rho_{nq},
\end{multline*}
and analogously,
\begin{equation*}
\int\limits_0^{\pi }\exp (-2i\rho x)P_-^\dagger(x)P_+(x)\,dx=O\left(n^{-1}\right),\quad\rho=\rho_{nq}.
\end{equation*}
Now it is clear that \eqref{symeq} follows from \eqref{smeq1}.
\end{proof}

\begin{lem} \label{lem:real}
Starting from some number $N$, all the eigenvalues $\rho_{nq}$, $|n| \ge N$, $q = \overline{1, m}$, are real.
\end{lem}

\begin{proof}
Let $\rho_0$ be an eigenvalue and $Y(x,\rho_0)$ be a corresponding eigenvector for the pencil $L$.
Consider the scalar product in $L_2((0, \pi), \mathbb{C}^m)$:
\begin{equation*}
 \mathcal{S}[Y] := \Bigl( Y,Y^{\prime\prime}+\left(\rho_0^2I+2i\rho_0Q_1+Q_0\right)Y\Bigr)=
\left(Y,Y^{\prime\prime}\right)+\rho_0^2(Y,Y)+2i\rho_0(Y,Q_1Y)+(Y,Q_0Y)=0.
\end{equation*}
Integrating by parts, we get
\begin{equation*}
\left(Y,Y^{\prime\prime}\right)=-Y^\dagger(\pi)(i\rho_0H_1+H_0)Y(\pi)+Y^\dagger(0)(i\rho_0h_1+h_0)Y(0)-\left(Y^{\prime},Y^{\prime}\right).
\end{equation*}

Every eigenvector can be represented in the form
\begin{equation} \label{repeigen}
Y(x, \rho_0)=\vv(x,\rho_0)C_\vv(\rho_0),
\end{equation}
where  $C_\vv(\rho_0)$ is a constant column vector, it can be chosen so that $\|C_\vv(\rho_0)\|=1$.

It follows from \eqref{asymptphi} and \eqref{repeigen}, that
when $|\rho_0|\rightarrow\infty$ by the set of eigenvalues,
\begin{equation} \label{asymptY}
	Y(x,\rho_0)=O(1),
\end{equation}
uniformly with respect to $x \in [0, \pi]$.

Let $\rho_0=\sigma+i\tau,\ \tau\not=0$.
Then
\begin{equation*}
\mbox{Im}\, \mathcal{S}[Y] = 2\sigma\tau(Y,Y)+2\tau(Y,Q_1Y)-\tau Y^\dagger(\pi)i H_1Y(\pi)+\tau Y^\dagger(0)i h_1Y(0)=0,
\end{equation*}
\begin{equation} \label{formsigma}
\sigma=-\frac{2(Y,Q_1Y)-Y^\dagger(\pi)i H_1Y(\pi)+Y^\dagger(0)i h_1Y(0)}{(Y,Y)}.
\end{equation}
It follows from \eqref{asymptY}, that the numerator of this fraction is bounded. Lemma~\ref{lem:symeq}
helps us to prove that the denominator  is bounded from below. Indeed, by positiveness of
the matrix $(I + h_1^{\dagger} h_1)$, we have
$$
 	(X, (I + h_1^{\dagger} h_1) X) \ge \al \| X \|^2, \quad \al > 0.
$$
Hence
$$
	(Y, Y) = \displaystyle\int\limits_0^{\pi }C_\vv^\dagger\vv^\dagger(x,\rho_0)\vv(x,\rho)C_\vv\,dx\ge\al\|C_\vv\|^2
$$
for sufficiently large $|\rho_0|$.

In fact, we have proved that the real parts $\sigma$ of eigenvalues are bounded as $|\rho_0| \to \infty$.
But this contradicts the asymptics from Lemma~\ref{lem:asymptrho}.
Thus, $\tau = 0$ for all eigenvalues, sufficiently large by their absolute value.
\end{proof}

\begin{remark} \label{rem:real}
Suppose
\begin{equation*}
  F[Y] := (Y', Y') - (Y, Q_0 Y) + Y^{\dagger} (\pi) H_0 Y(\pi) - Y^{\dagger}(0) h_0 Y(0) \ge 0
\end{equation*}
for all eigenvectors $Y = Y(x, \rho_0)$. For example, this holds when $Q_0(x) \le 0$, $H_0 \ge 0$,
$h_0 \le 0$.

Let $G[Y]$ be the numerator of the fraction in \eqref{formsigma}. Then
$$
 	\mbox{Re}\, \mathcal{S}[Y] = (\sigma^2 - \tau^2)(Y, Y) + \sigma G[Y] - F[Y] = 0.
$$
Note that by \eqref{formsigma}, $\sigma G[Y] = -2 \sigma^2 (Y, Y)$. Hence
$$
 	F[Y] = -(\sigma^2 + \tau^2)(Y, Y) \le 0.
$$
We conclude $\sigma^2 + \tau^2 = 0$, so in this special case, {\it all eigenvalues are real}.
Moreover, we note that \eqref{formsigma} can hold only for $\rho_0 = 0$.
\end{remark}

It is easy to check that
\begin{equation} \label{Phiexp}
    \Phi(x, \rho) = S(x, \rho) + \vv(x, \rho) M(\rho),
\end{equation}
\begin{equation} \label{Mexp}
    M(\rho) = -(V(\vv))^{-1} V(S).
\end{equation}
Consequently, $M(\rho)$ is a meromorphic function and its poles coincide with the zeros of $\det V(\vv)$,
i.e. the eigenvalues $\{ \rho_{nq} \}$.

\begin{lem}\label{lem:simplepoles}
The poles of the matrix function $M(\rho)$ at $\rho=\rho_{nq}$ for sufficiently large $|n|$ are simple.
\end{lem}

\begin{proof}

We will prove that the matrix function $(V(\vv))^{-1}$ has only simple poles at
sufficiently large eigenvalues, using the following well-known fact (see \cite[Lemma 2.2.1]{AM60}):

\smallskip

{\it
If there do not exist two nonzero vectors $a$ and $b$, such that
\begin{equation} \label{lemAM}
\begin{array}{c}
V(\vv)a=0, \\
\frac{d}{d\rho}V(\vv)a+V(\vv)b=0,
\end{array}
\end{equation}
at some point $\rho_0$, then $\rho_0$ is a simple pole of $(V(\vv))^{-1}$
}

\smallskip

In order to use this fact, we derive a formula for $\frac{d}{d \rho} V(\vv)$.

In this proof, we assume that $\rho$ is real. Let $\psi(x, \rho) = [\psi_{jk}(x, \rho)]_{j, k = \overline{1, m}}$
be a matrix solution of equation \eqref{eqv} under the initial conditions
\begin{equation*}
\psi(\pi,\rho)=I, \quad V(\psi)=0.
\end{equation*}
In view of (II), $\psi^{\dagger}(x, \rho)$ is a solution of the following equation
\begin{equation} \label{eqvZ}
 	Z'' + Z (\rho^2 I + 2 i \rho Q_1(x) + Q_0(x)) = 0,
\end{equation}
and consequently, the Wronskian $\langle\psi^\dagger(x,\rho),\vv(x,\rho)\rangle$ does not depend on $x$.
Therefore we calculate
\begin{equation*}
\langle\psi^\dagger(x,\rho),\vv(x,\rho)\rangle=\langle\psi^\dagger(x,\rho),\vv(x,\rho)\rangle_{\left|x=0\right.}=\bigl(U(\psi)\bigr)^\dagger,
\end{equation*}
\begin{equation*}
\langle\psi^\dagger(x,\rho),\vv(x,\rho)\rangle=\langle\psi^\dagger(x,\rho),\vv(x,\rho)\rangle_{\left|x=\pi\right.}=-V(\vv),
\end{equation*}
\begin{equation}\label{UV}
    \bigl(U(\psi)\bigr)^\dagger =-V(\vv).
\end{equation}

Let $\rho_0$ be a real eigenvalue of $L$. Then
\begin{equation*}
\frac{d}{dx}\langle\psi^\dagger(x,\rho),\vv(x,\rho_0)\rangle=(\rho_0-\rho)\psi^\dagger(x,\rho)\bigl((\rho+\rho_0)I+2i Q_1(x)\bigr)\vv(x,\rho_0),
\end{equation*}
\vspace{-1cm}
\begin{multline*}
\left.\langle\psi^\dagger(x,\rho),\vv(x,\rho_0)\rangle\right|_0^\pi=(\rho_0-\rho)H_1\vv(\pi,\rho_0)\\
-V\bigl(\vv(x,\rho_0)\bigr)-\Bigl(U\bigl(\psi(x,\rho)\bigr)\Bigr)^\dagger -(\rho_0-\rho)\psi^\dagger(0,\rho)h_1.
\end{multline*}

By \eqref{UV} we have
\begin{multline*}
V\bigl(\vv(x,\rho)\bigr)-V\bigl(\vv(x,\rho_0)\bigr)=(\rho_0-\rho)\Biggl[\psi^\dagger(0,\rho)h_1-H_1\vv(\pi,\rho_0)\\
+\int\limits_0^{\pi }\psi^\dagger(x, \rho)\bigl((\rho+\rho_0)I+2i Q_1(x)\bigr)\vv(x,\rho_0)\,dx\Biggr].
\end{multline*}
Taking a limit as $\rho$ tends to $\rho_0$, we get
\begin{equation} \label{Vderiv}
\frac{d}{d\rho}V(\vv)_{|\rho=\rho_0} =-\Biggl[\psi^\dagger(0,\rho_0)h_1-H_1\vv(\pi,\rho_0)+
2\int\limits_0^{\pi }\psi^\dagger(x,\rho_0)\bigl(\rho_0 I +i Q_1(x)\bigr)\vv(x,\rho_0)\,dx\Biggr].
\end{equation}

Every eigenvector $Y(x, \rho_0)$, corresponding to the eigenvalue $\rho_0$, has the form
\begin{equation*}
Y(x, \rho_0) =\vv(x,\rho_0)C_\vv=\psi(x,\rho_0)C_\psi,
\end{equation*}
where $C_{\vv}$ and $C_{\psi}$ are some constant vectors. We can assume that $\| C_{\vv} \| = 1$.
Using \eqref{Vderiv}, we obtain
\begin{multline} \label{smeq2}
C_\psi^\dagger\frac{d}{d\rho}V(\vv)_{|\rho=\rho_0}C_\vv=-\Biggl[Y^\dagger(0,\rho_0)h_1C_\vv-C_\psi^\dagger H_1Y(\pi,\rho_0)+\\
+ 2\int\limits_0^{\pi }Y^\dagger(x,\rho)i Q_1Y(x,\rho)\,dx+2\rho_0\int\limits_0^{\pi }C_\vv^\dagger\vv^\dagger(x,\rho_0)\vv(x,\rho_0)C_\vv\,dx\Biggr]
\end{multline}
If $\|C_\vv\|=1$, then $Y(x,\rho_0)=O(1),\ |\rho_0|\rightarrow\infty$ and $C_\psi=\vv(\pi,\rho_0)C_\vv$ is also bounded.
By Lemma~\ref{lem:symeq}, the integral in the last term is positive and bounded from below for sufficiently large $|\rho_0|$.
Therefore,
$C_\psi^\dagger\frac{d}{d\rho}V(\vv)C_\vv$ tends to $\pm\infty$
as $\rho_0=\rho_{n q},\ n\rightarrow\pm\infty$.

If $a$ and $b$ satisfy \eqref{lemAM}, then $a$ gives an eigenvector, i.e. $a = C_{\vv}$.
On one hand, from the second relation of \eqref{lemAM} we get
$$
  C_\psi^\dagger V(\vv)b = -C_\psi^\dagger\frac{d}{d\rho}V(\vv)C_\vv,
$$
that is unbounded for sufficiently large $|\rho_0|$.
On the other hand, by \eqref{UV} we have $C_\psi^\dagger V(\vv)=-C_\psi^\dagger\bigr(U(\psi)\bigl)^\dagger=0$.
We arrive at a contradiction, so for sufficiently large $|\rho_0|$, the vectors $a$ and $b$ in \eqref{lemAM} must be zero and
the poles of $(V(\vv))^{-1}$ are simple.
\end{proof}

\begin{remark} \label{rem:simple}
Consider the case from Remark~\ref{rem:real}. Note that if $\rho_0$ is real, then
the right hand side of \eqref{smeq2} equals zero if and only if \eqref{formsigma} holds.
But according to Remark~\ref{rem:real}, \eqref{formsigma} can be valid only for $\rho_0 = 0$.
Moreover, if $F[Y] > 0$ for all eigenvectors belonging to the zero eigenvalue, then \eqref{formsigma}
never holds. Following the proof of Lemma~\ref{lem:simplepoles}, we conclude that in this special case
all the poles of $M(\rho)$ are simple (if $F[Y]$ can be zero for $\rho_0 = 0$, then all the poles except
$\rho_0 = 0$ are simple).
\end{remark}

Introduce the notation: $J_q=\{ s\colon \om_s=\om_q\}$, $I_q=[I_{q,j k}]_{j, k=\overline{1,m}}$,
\begin{equation*}
 I_{q,j k}=
\begin{cases}
1, &j=k\ \text{and}\ \om_j\in J_q; \\
0,  &\text{otherwise}.
\end{cases}
\end{equation*}

Let $m_{n q}$ be the multiplicity of the eigenvalue $\rho_{n q}$
(as a zero of the analytic function $\det V(\vv)$).
We consider the data $\al_{n q} :=\Res\limits_{\rho=\rho_{n q}}M(\rho)$ and $m_{n q}$ only for sufficiently large $|n|$,
for which the assertions of Lemmas~\ref{lem:asymptrho},\ref{lem:real} and \ref{lem:simplepoles} hold,
so here we do not deal with non-simple poles.

\begin{lem}
The ranks of the weight matrices $\al_{nq}$ are equal to $m_{nq}$ for all such $n, q$ that the corresponding poles of
$M(\rho)$ are simple.
\end{lem}

The proof is similar to the one for the matrix Sturm-Liouville operator \cite[Lemma 4]{Bond11}.

In the following lemma, we split the eigenvalues into groups according to their asymptotics and sum all residues corresponding to a group.

\begin{lem} \label{lem:asymptalpha}
The following relation holds
\begin{equation*}
\sum_{s\in J_q}\frac{\al_{n s}}{m_{n s}}=\frac{1}{\pi n}(I-h_1)^{-1}W^\dagger I_q W(I+h_1)^{-1}+O\left(n^{-2}\right),\quad |n|\rightarrow\infty,\quad q=\overline{1,m},
\end{equation*}
where $W$ is the unitary matrix, defined in \eqref{initAWM}.
\end{lem}

\begin{proof}
Using \eqref{asymptS}, we obtain
\begin{equation*}
V(S)=\frac{1}{2}\exp(i\rho\pi)(I+H_1)P_-(\pi)+\frac{1}{2}\exp(-i\rho\pi)(I-H_1)P_+(\pi)+O\Bigl(|\rho|^{-1}\exp\bigl(|\tau|\pi\bigr)\Bigr),
\end{equation*}
Substitute this and \eqref{asymptVphi} into \eqref{Mexp}:
\begin{multline*}
M(\rho)=-(V(\vv))^{-1} V(S) =
-(i\rho)^{-1}\biggl\{\frac{1}{2}\exp(i\rho\pi)(I+H_1)P_-(\pi)(I-h_1)\\
-\frac{1}{2}\exp(-i\rho\pi)(I-H_1)P_+(\pi)(I+h_1)+O\Bigl(|\rho|^{-1}\exp\bigl(|\tau|\pi\bigr)\Bigr)\biggr\}^{-1}\\
\cdot\biggl\{\frac{1}{2}\exp(i\rho\pi)(I+H_1)P_-(\pi)+\frac{1}{2}\exp(-i\rho\pi)(I-H_1)P_+(\pi)\\
+O\Bigl(|\rho|^{-1}\exp\bigl(|\tau|\pi\bigr)\Bigr)\biggr\}=-(i\rho)^{-1}(I-h_1)^{-1}\biggl\{\exp(2i\rho\pi)I \\
-P_-^{-1}(\pi)(I+H_1)^{-1}(I-H_1)P_+(\pi)(I+h_1)(I-h_1)^{-1}+O\Bigl(|\rho|^{-1}\exp\bigl(2|\tau|\pi\bigr)\Bigr)\biggr\}\\
\cdot\biggl\{\exp(2i\rho\pi)I +P_-^{-1}(\pi)(I+H_1)^{-1}(I-H_1)P_+(\pi)+ O\Bigl(|\rho|^{-1}\exp\bigl(2|\tau|\pi\bigr)\Bigr)\biggr\}.
\end{multline*}
Finally, we get
\begin{multline} \label{asymptM}
M(\rho)=-(i\rho)^{-1}(I-h_1)^{-1}\biggl\{\exp(2i\rho\pi)I-A +O\Bigl(|\rho|^{-1}\exp\bigl(2|\tau|\pi\bigr)\Bigr)\biggr\}^{-1}\\
\cdot\biggl\{\exp(2i\rho\pi)I+B+O\Bigl(|\rho|^{-1}\exp\bigl(2|\tau|\pi\bigr)\Bigr)\biggr\},
\end{multline}
where $A$ was defined by \eqref{defA} and
\begin{equation} \label{defB}
B=A(I-h_1)(I+h_1)^{-1}.
\end{equation}
Introduce the matrix function
\begin{equation*}
M_0(\rho)=-(i\rho)^{-1}(I-h_1)^{-1}\bigl(\exp(2i\rho\pi)I-A\bigr)^{-1}\cdot\bigl(\exp(2i\rho\pi)I+B\bigr).
\end{equation*}

Let $\ga_{nq}$ be the contours used in the proof of Lemma 1.
They are circles with centers in $\rho_{nq}^0 = n + \om_q$ and a sufficiently small fixed radius, such that
$\rho_{ms}^0 \in \mbox{int}\, \ga_{nq}$ if and only if $\rho_{ms}^0 = \rho_{nq}^0$.

According to \eqref{asymptM}, we have
\begin{equation*}
\bigl\|M(\rho)-M_0(\rho)\bigr\|\le\frac{C}{n^2},\quad\rho\in\ga_{nq}.
\end{equation*}
The residue theorem yields
\begin{equation} \label{intM}
\frac{1}{2\pi i}\int\limits_{\ga_{n q}}\bigl(M(\rho)-M_0(\rho)\bigr)\,d\rho=\sum_{s\in J_q}\frac{1}{m_{n s}}\Res_{\rho=\rho_{n q}}M(\rho)-\Res_{\rho=\rho_{n q}^0}M_0(\rho)=O(n^{-2}),\quad |n|\to\infty.
\end{equation}

Clearly, $M_0(\rho)$ has simple poles at $\rho_{n q}^0$, and using \eqref{initAWM}, we derive
\begin{multline*}
\Res_{\rho=\rho_{n q}^0}M_0(\rho)=-\lim_{\rho\to\rho_{n q}^0}(\rho-\rho_0)(i\rho)^{-1}(I-h_1)^{-1}W^\dagger\bigl(\exp(2i\rho\pi )I-\mathcal{M}\bigr)^{-1}W\bigl(\exp(2i\rho\pi )I+B\bigr) \\
=\frac{1}{2\pi \rho_{n q}^0}(I-h_1)^{-1}W^\dagger\mu_q^{-1}I_q W(\mu_q I+B)=\colon\al_{n q}^0.
\end{multline*}
Using the obvious relation $\mu_q^{-1}I_q\mathcal{M}=I_q$ and \eqref{defB}, we get
\begin{equation*}
W^\dagger\mu_q^{-1}I_q W B(I+h_1)=W^\dagger\mu_q^{-1}I_q WW^\dagger\mathcal{M}W(I-h_1)=W^\dagger I_q W(I-h_1).
\end{equation*}
Therefore
\begin{equation*}
(I-h_1)\al_{n q}^0(I+h_1)=\frac{1}{2\pi \rho_{n q}^0}\bigl(W^\dagger I_q W(I+h_1)+W^\dagger I_q W(I-h_1)\bigr),
\end{equation*}
\begin{equation*}
\al_{n q}^0=\frac{1}{\pi n}(I-h_1)^{-1}W^\dagger I_q W(I+h_1)^{-1}+O(n^{-2}), \quad |n| \to \infty.
\end{equation*}
Together with \eqref{intM} this yields the assertion of the Lemma.
\end{proof}

Summarizing the results of the Lemmas, we arrive at Theorem~\ref{thm:SD}.

Note that there is a sufficiently wide class of pencils, for which
assertions (ii) and (iii) of Theorem~\ref{thm:SD} are valid for all $n \in w$.
Namely, (ii) holds for all eigenvalues when
\begin{equation*}
  F[Y] := (Y', Y') - (Y, Q_0 Y) + Y^{\dagger} (\pi) H_0 Y(\pi) - Y^{\dagger}(0) h_0 Y(0) \ge 0
\end{equation*}
for all eigenvectors $Y = Y(x, \rho_0)$.
For example, this holds when $Q_0(x) \le 0$, $H_0 \ge 0$,
$h_0 \le 0$.
In this case, (iii) is valid for all eigenvalues except zero. To make it valid for zero,
we require additionally that $F[Y] > 0$ for eigenvectors corresponding to $\rho_0 = 0$.

\bigskip

{\large \bf 3. Contour integration}

\medskip

Assume that the spectral data $SD$ of an unknown pencil $L$ are given.

At first, choose a model pencil $\tilde L$ with the spectral data satisfying asymptotics
\eqref{asymptrho} and \eqref{asymptalpha} with the same constants as the data of $L$.
This can be done by the following algorithm.

\medskip

\textbf{Algorithm 1.} \textit{Choice of a model pencil}.

\begin{enumerate}

\item Construct $M(\rho)$ via \eqref{reprM}.
Using \eqref{asymptM}, one can easily derive
\begin{equation*}
M(\rho)=(i \rho)^{-1}(I+h_1)^{-1}+O(\rho^{-2}),\quad \rho=i \tau,\quad \tau \to+\infty.
\end{equation*}
Therefore, having $M(\rho)$, we can obtain $h_1$ and set $\tilde h_1\colon =h_1$.

\item Find $\om_q$ from the asymptotics \eqref{asymptrho}, $\mu_q\colon =e^{2\pi i \om_q},\ \mathcal{M}:=\mbox{diag} \, \{\mu_q\}_{q= \overline{1, m}}.$

\item Construct $J_q$ and $I_q$. From the asymptotics \eqref{asymptalpha}, using $h_1$, find $W^\dagger I_q W$.
Note that
$$
	A=W^\dagger\mathcal{M}W=\sum\limits_{q=1}^m|J_q|^{-1}\mu_q W^\dagger I_q W,
$$
so we can construct $\tilde A=A$.

\item Put $\tilde H_1=0$. According to \eqref{defA},
we need to find such a potential $\tilde Q_1(x)$, that
\begin{equation*}
\tilde P_-^\dagger(\pi)\tilde P_+(\pi)=\tilde A(I-\tilde h_1)(I+\tilde h_1)^{-1}=:\mathcal{O}.
\end{equation*}
Try a constant anti-Hermitian matrix:
\begin{equation*}
\tilde Q_1(x)\equiv T,\quad T^\dagger=-T.
\end{equation*}
Then $P_{\pm}(x) = \exp(\pm Tx)$, and we arrive at the matrix equation $\exp(2 T \pi) = \mathcal{O}$.
It can be easily solved, since the matrix $\mathcal{O}$ is unitary.

\item Put $\tilde Q_0(x) = \tilde h_0 = \tilde H_0 = 0$, $x \in [0, \pi]$.

\end{enumerate}

Denote
\begin{equation} \label{defD}
 	D(x, \rho, \theta) = \frac{\langle \vv^*(x, \theta), \varphi(x, \rho) \rangle}{\rho - \theta}.
\end{equation}
Similarly to the scalar case (see \cite[Lemma 3]{Yur00}), one can obtain the following estimate
\begin{equation} \label{estD}
 	\| D(x, \rho, \theta) \| \le C_x \frac{|\rho| + |\theta| + 1}{|\rho - \theta| + 1}
 	\exp(|\mbox{Im}\,\rho|x) \exp(|\mbox{Im}\,\theta|x),
\end{equation}
where $C_x$ is a constant depending on $x$.

Consider the block-matrix $\mathcal{P}(x,\rho)=[\mathcal{P}_{jk}(x,\rho)]_{j,k=1,2}$
defined by
\begin{equation} \label{defP}
\mathcal{P}(x,\rho) \left[ \begin{array}{ll} \tilde\vv(x,\rho) & \tilde\Phi(x,\rho)\\ \tilde\vv'(x,\rho) & \tilde\Phi'(x,\rho) \end{array}\right]
= \left[\begin{array}{ll} \vv(x,\rho) & \Phi(x,\rho)\\ \vv'(x,\rho) & \Phi'(x,\rho) \end{array}\right].
\end{equation}

Note that if $Y$ satisfies \eqref{eqv} and $Z$ satisfies \eqref{eqvZ} for the same $\rho$, then
\begin{equation} \label{wron}
  \frac{d}{dx} \langle Z, Y\rangle = 0.
\end{equation}
Hence the Wronskian $\langle Z, Y\rangle$ does not depend on $x$. Consequently, using
boundary conditions for $\vv$, $\vv^*$, $\Phi$ and $\Phi^*$, we easily derive
\begin{equation} \label{inverseform}
\left[\begin{array}{ll} \vv(x,\rho) & \Phi(x,\rho)\\ \vv'(x,\rho) & \Phi'(x,\rho) \end{array}\right]^{-1} =
\left[\begin{array}{ll} {\Phi^*}'(x, \rho) & -\Phi^*(x,\rho)\\ -{\vv^*}'(x,\rho) & \vv^*(x,\rho) \end{array}\right],
\end{equation}
\begin{equation} \label{Pj12}
\begin{array}{l}
\mathcal{P}_{j1}(x,\rho)=\vv^{(j-1)}(x,\rho){\tilde\Phi^{*'}}(x,\rho)-
\Phi^{(j-1)}(x,\rho){\tilde\vv^{*'}}(x,\rho), \\
\mathcal{P}_{j2}(x,\rho)=\Phi^{(j-1)}(x,\rho){\tilde\vv}^{*}(x,\rho)-
\vv^{(j-1)}(x,\rho){\tilde\Phi}^{*}(x,\rho).
\end{array}
\end{equation}
We mention two other facts that follows from \eqref{wron} and will be used further:
\begin{equation} \label{wronphi}
  \langle \vv^*(x, \rho), \vv(x, \rho) \rangle = \langle \vv^*(x, \rho), \vv(x, \rho) \rangle_{x = 0} = 0,
\end{equation}
\begin{equation} \label{eqMM*}
	\langle \Phi^*(x, \rho), \Phi(x, \rho) \rangle = M(\rho) - M^*(\rho) = 0,
	\quad \Rightarrow \quad M(\rho) \equiv M^*(\rho).
\end{equation}

Denote
$$
	G_{\de} := \{ \rho \in \mathbb{C} \colon |\rho - \rho_{nq} | > \de, \, |\rho - \tilde \rho_{nq}| > \de, \, n \in w, \, q = \overline{1, m} \}, \quad \de > 0.
$$

In a standard way (see \cite{BF13}), one can prove the following Lemma.

\begin{lem} \label{lem:asymptP}
Assume $h_1 = \tilde h_1$.
Then for each fixed $x \in [0, \pi]$,
the entries of the block-matrix $\mathcal{P}$ have the following asymptotics as $|\rho| \to \infty$,
$\rho \in G_{\de}$:
$$
  \begin{array}{ll}
   \mathcal{P}_{11}(x, \rho) = \Omega(x) + O(\rho^{-1}), & \mathcal{P}_{12}(x, \rho) = \rho^{-1} \Lambda(x) + O(\rho^{-2}), \\
   \mathcal{P}_{21}(x, \rho) = - \rho \Lambda(x) + O(1), & \mathcal{P}_{22}(x, \rho) = \Omega(x) + O(\rho^{-1}),
  \end{array}
$$
where
\begin{equation} \label{defOmega}
    \Omega(x) := \frac{1}{2} \left( P_{-}(x) \tilde P^{\dagger}_{-}(x) + P_{+}(x) \tilde P^{\dagger}_{+}(x) \right),
\end{equation}
\begin{equation} \label{defLambda}
    \Lambda(x) := \frac{1}{2 i} \left( P_{-}(x) \tilde P^{\dagger}_{-}(x) - P_{+}(x) \tilde P^{\dagger}_{+}(x) \right).
\end{equation}
\end{lem}

For simplicity, we assume that $N = 0$ and $\tilde N = 0$, i.e. the assertion of Theorem~\ref{thm:SD} holds 
for all the eigenvalues of both pencils $L$ and $\tilde L$. We return to the general case in Section 6, since it requires
some modifications.

\begin{lem} \label{lem:contour}
The following relations hold:
\begin{equation} \label{cont1}
 	\Omega(x) \tilde \vv(x, \rho) =  \vv(x, \rho) +
 	\sum_{n \in w} \sum_{q = 1}^m \Bigl\{ \vv(x, \rho_{nq}) m_{nq}^{-1} \al_{nq} \tilde D(x, \rho, \rho_{nq}) -
 	\vv(x, \tilde \rho_{nq}) \tilde m_{nq}^{-1} \tilde \al_{nq} \tilde D(x, \rho, \tilde \rho_{nq}) \Bigr\},
\end{equation}
\begin{multline} \label{cont2}
 	D(x, \rho, \theta) - \tilde D(x, \rho, \theta) - \vv^*(x, \theta) \Lambda(x) \tilde \vv(x, \rho)
 	+ \sum_{n \in w} \sum_{q = 1}^m \Bigl\{
 	 D(x, \rho_{nq}, \theta) m_{nq}^{-1} \al_{nq} \tilde D(x, \rho, \rho_{nq}) \\ -
 	 D(x, \tilde \rho_{nq}, \theta) \tilde m_{nq}^{-1} \tilde \al_{nq} \tilde D(x, \rho, \tilde \rho_{nq})
 	\Bigr\} = 0.
\end{multline}
\end{lem}

\begin{proof}
1. Choose a number $b > \sup\limits_{n, q} \max \{ |\mbox{Im}\, \rho_{nq}|, |\mbox{Im}\, \tilde \rho_{nq}| \}$. Denote
$$
 	\Sigma_R := \{ \rho \colon |\mbox{Im}\, \rho| < b, \, |\rho| < R \}, \quad \ga_R := \partial \Sigma_R, \quad
 	\Gamma_R := \{ \rho \colon |\rho| = R \}, \quad \gamma_R^0 = \Gamma_R - \ga_R
$$
(the defined contours have counterclockwise circuit).
By Cauchy's integral formula,
$$
  \mathcal{P}_{1k}(x, \rho) - \Omega(x) \delta_{1k} = \frac{1}{2 \pi i}
  \int_{\ga_R^0} \frac{\mathcal{P}_{1k}(x, \theta) - \Omega(x) \delta_{1k}}{\rho - \theta} \, d \theta, \quad
  k = 1, 2, \quad x \in [0, \pi], \quad \rho \in \mbox{int}\, \ga_R^0,
$$
where $\de_{jk}$ is the Kronecker delta. Put
$$
\eps_R(x, \rho) := \int_{\Gamma_R} \frac{\mathcal{P}_{1k}(x, \theta) - \Omega(x) \delta_{1k}}{\rho - \theta} \, d \theta,
$$
Since we have chosen $\tilde h_1 = h_1$, we can apply Lemma~\ref{lem:asymptP} and obtain
$\lim\limits_{\substack{R \to \infty \\ \Gamma_R \subset G_{\de} }} \eps_R(x, \rho) = 0$.
Consequently,
\begin{equation} \label{smeq3}
 	\mathcal{P}_{1k}(x, \rho) = \Omega(x) \delta_{1k} + \lim_{\substack{R \to \infty \\ \Gamma_R \subset G_{\de}}}
 	\frac{1}{2 \pi i} \int_{\ga_R} \frac{\mathcal{P}_{1k}(x, \theta)}
 	{\rho - \theta} \, d \theta, \quad k = 1, 2, \quad x \in [0, \pi], \quad |\mbox{Im}\, \rho| \ge b.
\end{equation}

Substituting \eqref{smeq3} into \eqref{defP} and using \eqref{Pj12}, \eqref{Phiexp} and \eqref{defD}, we obtain
$$
	\vv(x, \rho) = \Omega(x) \tilde \vv(x, \rho) - \lim_{\substack{R \to \infty \\ \Gamma_R \subset G_{\de}}}
	\frac{1}{2 \pi i} \int_{\ga_R} \vv(x, \theta) (M(\theta) - \tilde M(\theta)) \tilde D(x, \rho, \theta) \, d \theta,
$$
since the terms with $S(x, \theta)$ and $\tilde S^*(x, \theta)$ vanish by Cauchy's theorem.
Calculating the integral over $\ga_R$ by the residue theorem, we arrive at \eqref{cont1}.

2. Analogously to \eqref{smeq3}, for $x \in [0, \pi]$, $|\mbox{Im}\,\rho|, |\mbox{Im} \, \theta| > b$
and $\ga_R \subset G_{\de}$, we derive
$$
 \frac{\mathcal{P}_{21}(x, \rho) - \mathcal{P}_{21}(x, \theta)}{\rho - \theta} = -\Lambda(x) + \frac{1}{2 \pi i} \int_{\ga_R}
 \frac{\mathcal{P}_{21}(x, \xi)}{(\rho - \xi)(\xi - \theta)} \, d \xi + \eps_{R, 21}(x, \rho, \theta),
$$
$$
  \frac{\mathcal{P}_{jk}(x, \rho) - \mathcal{P}_{jk}(x, \theta)}{\rho - \theta} = \frac{1}{2 \pi i} \int_{\ga_R}
  \frac{\mathcal{P}_{jk}(x, \xi)}{(\rho - \xi)(\xi - \theta)} \, d \xi + \eps_{R, jk}(x, \rho, \theta), \quad (j, k) \ne (2, 1),
$$
where $\lim\limits_{\substack{R \to \infty \\ \Gamma_R \subset G_{\de}}} \eps_{R, jk}(x, \rho, \theta) = 0$, $j, k = 1, 2$.

For an arbitrary vector function $Y = Y(x)$, in view of \eqref{inverseform} and \eqref{defP}, we have
$$
	\mathcal{P} \left[ \begin{array}{l} Y \\ Y' \end{array} \right] =
	\left[ \begin{array}{l} \vv \\ \vv' \end{array} \right] \langle \tilde \Phi^*, Y \rangle -
	\left[ \begin{array}{l} \Phi \\ \Phi' \end{array} \right] \langle \tilde \vv^*, Y \rangle.
$$
Therefore
\begin{multline} \label{sum1}
    [{\vv^*}'(x, \theta), -\vv^*(x, \theta)]\frac{\mathcal{P}(x, \rho) - \mathcal{P}(x, \theta)}{\rho - \theta}
    \left[ \begin{array}{l} \tilde \vv(x, \rho) \\ \tilde \vv'(x, \rho) \end{array} \right] =
    \vv^*(x, \theta) \Lambda(x) \tilde \vv(x, \rho) \\ +
	\frac{1}{2 \pi i} \int_{\ga}
    \frac{(\langle \vv^*(x, \theta), \vv(x, \xi) \rangle \langle \tilde \Phi^*(x, \xi), \tilde \vv(x, \rho) \rangle -
    \langle \vv^*(x, \theta), \Phi(x, \xi) \rangle \langle \tilde \vv^*(x, \xi), \tilde \vv(x, \rho) \rangle ) d \xi}{(\rho - \xi)(\xi - \theta)} \\
    + \eps^1_R(x, \rho, \theta).
\end{multline}
Using \eqref{defP} and \eqref{inverseform}, we obtain
\begin{equation} \label{sum2}
  [{\vv^*}'(x, \theta), -\vv^*(x, \theta)] \mathcal{P}(x, \rho)
  \left[ \begin{array}{l} \tilde \vv(x, \rho) \\ \tilde \vv'(x, \rho) \end{array} \right] =
  \langle \vv^*(x, \theta), \vv(x, \rho) \rangle,
\end{equation}
\begin{equation} \label{sum3}
  [{\vv^*}'(x, \theta), -\vv^*(x, \theta)] \mathcal{P}(x, \theta)
  \left[ \begin{array}{l} \tilde \vv(x, \rho) \\ \tilde \vv'(x, \rho) \end{array} \right] =
  \langle \tilde \vv^*(x, \theta), \tilde \vv(x, \rho) \rangle.
\end{equation}

Combining \eqref{sum1}, \eqref{sum2} and \eqref{sum3} and using \eqref{defD}, we get
\begin{multline*}
 	D(x, \rho, \theta) - \tilde D(x, \rho, \theta)
 	- \vv^*(x, \theta) \Lambda(x) \tilde \vv(x, \rho) - \frac{1}{2 \pi i} \int_{\ga_R} \Bigl( D(x, \xi, \theta) \tilde M^*(\xi)
 	\tilde D(x, \rho, \xi) \\ - D(x, \xi, \theta) M(\xi) \tilde D(x, \rho, \xi) \Bigr) \, d \xi = \eps^1_R(x, \rho, \theta).
\end{multline*}
Note that $M(\rho) \equiv M^*(\rho)$ by \eqref{eqMM*} and $\lim_{R \to 0} \eps_R^1(x, \rho, \theta) = 0$.
Therefore, calculating the
integral over $\ga_R$ by the residue theorem, we obtain \eqref{cont2}.	
\end{proof}

Denote
$$
    \tilde \Omega(x) := \frac{1}{2} \left( \tilde P_{-}(x) P^{\dagger}_{-}(x) + \tilde P_{+}(x) P^{\dagger}_{+}(x) \right),
$$
$$
    \tilde \Lambda(x) := \frac{1}{2 i} \left( \tilde P_{-}(x) P^{\dagger}_{-}(x) - \tilde P_{+}(x) P^{\dagger}_{+}(x) \right),
$$
(compare with \eqref{defOmega}, \eqref{defLambda}).
Symmetrically with the relations \eqref{cont1} and \eqref{cont2}, we have
\begin{equation} \label{cont3}
 	\tilde \Omega(x) \vv(x, \rho) =  \tilde \vv(x, \rho)  -
 	\sum_{n \in w} \sum_{q = 1}^m \Bigl\{ \tilde \vv(x, \rho_{nq}) m_{nq}^{-1} \al_{nq} D(x, \rho, \rho_{nq}) -
 	\tilde \vv(x, \tilde \rho_{nq}) \tilde m_{nq}^{-1} \tilde \al_{nq} D(x, \rho, \tilde \rho_{nq}) \Bigr\},
\end{equation}
\begin{multline} \label{cont4}
 	D(x, \rho, \theta) - \tilde D(x, \rho, \theta)
 	+ \tilde \vv^*(x, \theta) \tilde \Lambda(x) \vv(x, \rho) +
 	\sum_{n \in w} \sum_{q = 1}^m \Bigl\{
 	 \tilde D(x, \rho_{nq}, \theta) m_{nq}^{-1} \al_{nq} D(x, \rho, \rho_{nq}) \\ -
 	 \tilde D(x, \tilde \rho_{nq}, \theta) \tilde m_{nq}^{-1} \tilde \al_{nq} D(x, \rho, \tilde \rho_{nq})
 	\Bigr\} = 0.
\end{multline}

\bigskip

{\large \bf 4. Construction of the main equation in a Banach space}

\bigskip

In the previous section, Inverse Problem 1 was reduced to the linear equation \eqref{cont1}.
However, the series in this equation does not converge absolutely, so it is inconvenient to use \eqref{cont1}
as a main equation of the inverse problem. 
Now we plan to transform \eqref{cont1} into an equation in a Banach space
of infinite sequences.

In contrast to the scalar case, the spectrum of the matrix pencil $L$ can contain subsequences of eigenvalues 
with the similar asymptotics (some of $\om_q$ in \eqref{asymptrho} can be equal).
In order to overcome this difficulty, we divide the eigenvalues into groups $G_k$ by asymptotics
and then operate not with single eigenvalues but with the whole groups.

In this section, we assume that $x \in [0, \pi]$ is fixed.

Introduce sets $G_{n q}= \{ \rho_{n s},\tilde\rho_{n s},\ s\in J_q\}, \ n \in w$.
If some $\om_q$ are equal to each other, the corresponding sets coincide.
Therefore we assume (without loss of generality),
that $\{ \om_q\}_{q=1}^p$ are all the different values among $\{ \om_q\}_{q=1}^m$.
Consider $G_{n q}, \ q= \overline{1, p}$ and renumerate them as  $G_k, \ k\in\mathbb{Z} \backslash \{0\}$,
in the natural order.

Let $G=\{g_i\}_{i=1}^r$ be a non-empty set of distinct numbers. Define the diameter of $G$:
\begin{equation*}
\mbox{diam}(G) :=\underset{i=\overline{2, r}}{\max}\bigl\{|g_i-g_1|\bigr\}
\end{equation*}
and a finite-dimensional space $\mathcal{B}(G) :=\bigl\{f\colon G\to\mathbb{C}^{m\times m}\bigr\}$ with the norm
\begin{equation*}
\|f\|_{\mathcal{B}(G)}=\max\biggl\{\bigl\|f(g_1)\bigr\|,\ \underset{i=\overline{2, r}}{\max}\Bigl\{\bigl\|f(g_1)-f(g_i)\bigr\|\Bigr\}\bigl(\mbox{diam}(G)\bigr)^{-1}\biggr\}
\end{equation*}
Note that $\mbox{diam}(G)$ can be zero only in the case $r=1$, but then $\|f\|_{\mathcal{B}(G)}=\|f(g_1)\|$.

Consider a Banach space
\begin{equation} \label{defB1}
\mathcal{B}=\bigl\{f=\{f_k\}_{k\in\mathbb{Z}\backslash\{0\}}\colon f_k\in\mathcal{B}(G_k),\ \|f\|_{\mathcal{B}} := \underset{k\in\mathbb{Z}\backslash\{0\}}{\sup}\|f_k\|_{\mathcal{B}(G_k)}<\infty\bigr\},
\end{equation}

By \eqref{asymptrho}, $\mbox{diam}(G_{n q})\le\dfrac{C}{|n|}$,
and consequently, $\mbox{diam}(G_k)\le\dfrac{C}{|k|}$.
Using \eqref{asymptphi} and Schwarz's Lemma, it is easy to see that
$$
 	\| \vv(x, \rho_1) \| \le C, \quad \| \vv(x, \rho_1) - \vv(x, \rho_2) \| \le C |\rho_1 - \rho_2|,
$$
for $\rho_1, \rho_2 \in G_k$.
Therefore $\vv(x, \rho)$ forms an element of $\mathcal{B}$:
\begin{equation*}
\vv(x,\rho)_{|\mathcal{B}}= \bigl\{\vv(x,\rho)_{|G_{n q}}\bigr\}_{n \in w, q=\overline{1, p}} \in \mathcal{B}
\end{equation*}
\begin{equation*}
\vv(x,\rho)_{|G_{n q}}=\bigl\{\vv(x,\rho_{n s}), \ \vv(x,\tilde\rho_{n s}), \ s \in J_q\bigr\}.
\end{equation*}
Denote $\psi(x) := \vv(x,\rho)_{|\mathcal{B}}$, $\tilde \psi(x) := \vv(x,\rho)_{|\mathcal{B}}$.
Then the relations \eqref{cont1} and \eqref{cont3}
can be transformed into the following equations in the Banach space $\mathcal{B}$:
\begin{equation} \label{eqpsi}
 	\Omega(x) \tilde \psi(x) = \psi(x) (I + \tilde R(x)),
\end{equation}
$$
 	\tilde \Omega(x) \psi(x) = \tilde \psi(x) (I - R(x)),
$$
where $I$ is the identity operator in $\mathcal{B}$, and $R(x)$, $\tilde R(x)$ are linear operators, acting from
$\mathcal{B}$ to $\mathcal{B}$.
\footnote{These operators belong to the ring for which the elements of $\mathcal{B}$ form a right
module, so we write operators to the right of operands.}
The explicit form of $\tilde R(x)$, $R(x)$ can be derived from \eqref{cont1} and \eqref{cont3}.
Further we investigate the operator $R(x)$, the same properties for $\tilde R(x)$ can be obtained symmetrically.

According to \eqref{cont3}, the operator $R(x)$ acts to an arbitrary element $\psi \in \mathcal{B}$ in the following way:
\begin{equation} \label{defRsimp}
\bigl(\psi R(x)\bigr)_n = \sum\limits_{|k|>0}\psi_k R_{k,n}(x),  \quad R_{k,n}(x)\colon\mathcal{B}(G_k)\to\mathcal{B}(G_n), 
\quad k, n \in \mathbb{Z} \backslash \{ 0\}.
\end{equation}

\begin{lem} \label{lem:Rbound}
The series in \eqref{defRsimp} converge absolutely and the operator $R(x)$ is bounded on $\mathcal{B}$.
\end{lem}

\begin{proof}

Consider two sets
$G_n=\{g_i\}_{i=1}^r$ and $G_k=\{ s_j \}_{j = 1}^l$ю The operator $R_{n,0}(x,\rho)$ acts in $\mathcal{B}(G_n)$ in the following way: 
\begin{equation} \label{smeq4}
 	h(g_i) = f R_{k, n}(x) = \sum\limits_{j=1}^l f(s_j)\al_j D(x,g_i,s_j).
\end{equation}
where $\al_j=m_{n q}^{-1}\al_{n q}$ for $g_j=\rho_{n q}$ and $\al_j=-\tilde m_{n q}^{-1}\tilde\al_{n q}$ for $g_j=\tilde\rho_{n q}$.
\footnote{Here, in order to represent the operator in a more convenient way, 
we consider $G_n$ as a multiset.
For example, if $\rho_{nq} = \tilde \rho_{nq}$, they correspond to two equal elements in $G_n$: 
$g_i$ and $g_j$, $f(g_i) = f(g_j)$. Clearly, this does not affect the norm $\| f \|_{\mathcal{B}(G_n)}$.}

Let us prove that
\begin{equation} \label{smeq5}
	h \in \mathcal{B}(G_n), \quad \| h \|_{\mathcal{B}(G_n)} \le \frac{C(|n| + |k|)}{k^2 (|n - k| + 1)} \| f \|_{\mathcal{B}(G_k)},
\end{equation}
Using \eqref{smeq4}, we get
\begin{multline*} 
h(g_1)=f(s_1)\left(\sum\limits_{j=1}^l\al_j\right) D(x,g_1,s_1)+f(s_1)\sum\limits_{j=2}^l\al_j\bigl(D(x,g_1,s_j)-D(x,g_1,s_1)\bigr) \\
+\sum\limits_{j=2}^l\bigl(f(s_j)-f(s_1)\bigr)\al_j D(x,g_1,s_j),
\end{multline*}
\vspace{-1cm}
\begin{multline*}
 	h(g_i) - h(g_1) = \sum_{j = 1}^l f(s_j) \al_j \left( D(x, g_i, s_j) - D(x, g_1, s_j) \right) \\ =
 	f(s_1) \left( \sum_{j = 1}^l \al_j \right) \left( D(x, g_i, s_1) - D(x, g_1, s_1) \right)  +
 	f(s_1) \sum_{j = 2}^l \al_j \bigl( D(x, g_i, s_j) - D(x, g_1, s_j) \\ - D(x, g_i, s_1) + D(x, g_1, s_1)
 	\bigr) + \sum_{j = 2}^l (f(s_j) - f(s_1)) \al_1 \left( D(x, g_i, s_j) - D(x, g_1, s_j)\right).
\end{multline*}
Since $f \in \mathcal{B}(G_k)$, we have
\begin{equation*}
\bigl\|f(s_1)\bigr\|\le\|f\|_{\mathcal{B}(G_k)},\quad\bigl\|f(s_j)-f(s_1)\bigr\|\le |k|^{-1} \|f\|_{\mathcal{B}(G_k)}.
\end{equation*} 
Using \eqref{asymptrho}, \eqref{asymptalpha}, \eqref{estD} and Schwarz's lemma (see \cite{FY01}), it is easy to obtain the
following estimates:
\begin{equation*}
\left\|\sum\limits_{j=1}^l\al_j\right\|\le\frac{C}{k^2},\quad\left\|\al_j\right\|\le\frac{C}{|k|},
\end{equation*}
\begin{equation*}
\bigl\| D(x, g_i, s_j) \bigr\| \le \frac{C(|n| + |k|)}{(|n - k| + 1)}, \quad
\bigl\| D(x, g_i, s_j) - D(x, g_i, s_1)\bigr\| \le \frac{C(|n| + |k|)}{|k|(|n - k| + 1)},  
\end{equation*}
\begin{equation*}
\bigl\|D(x,g_i,s_j)-D(x,g_1,s_j)\bigr\|\le \frac{C(|n| + |k|)}{|n|(|n - k| + 1)},
\end{equation*}
\begin{equation*}
\bigl\|  D(x, g_i, s_j) - D(x, g_1, s_j) - D(x, g_i, s_1) + D(x, g_1, s_1) \bigr\|\le \frac{C(|n| + |k|)}{|n| |k|(|n - k| + 1)}.
\end{equation*}
Combining these relations and using the definition of $\mathcal{B}(G_n)$--norm, we arrive at \eqref{smeq5}.

Substitute \eqref{smeq5} into \eqref{defRsimp}:
\begin{equation*}
\bigl\|\psi R(x) \bigr\| = \underset{n\in\mathbb{Z}}{\sup}\Bigl\|\bigl(\psi R(x)\bigr)_n\Bigr\| \le 
\| \psi \| \left( \sum_{|k| > 0} \frac{|n| + |k|}{k^2 (|n - k| + 1)} \right) \le C \| \psi \|.
\end{equation*}
Hence
$\bigl\| R(x) \bigr\|_{\mathcal{B} \to \mathcal{B}} < \infty.$
\end{proof}

Note that \eqref{cont4} is equivalent to the relation
\begin{equation} \label{fromcont4}
   R(x) - \tilde R(x) + \tilde \Theta(x) + \tilde R(x) R(x) = 0,	
\end{equation}
where the operator $\tilde \Theta(x) \colon \mathcal{B} \to \mathcal{B}$ for each fixed $x \in [0, \pi]$
is defined by the following formula:
\begin{equation} \label{defTheta}
 	\psi \tilde \Theta(x) = \Biggl[\sum_{n \in w} \sum_{q = 1}^m \bigl\{ \psi(\rho_{nq}) m^{-1}_{nq} \al_{nq}\tilde \vv^*(x, \rho_{nq}) -
 	\psi(\tilde \rho_{nq}) \tilde m^{-1}_{nq} \tilde \al_{nq} \tilde \vv^*(x, \rho_{nq})
 	\bigl\} \Biggr] \tilde \Lambda(x) \vv(x, \rho)_{|\mathcal{B}}.
\end{equation}

\begin{thm} \label{thm:homo}
For each fixed $x \in [0, \pi]$, the subspace of solutions of the homogeneous equation
\begin{equation} \label{homo}
\beta (I + \tilde R(x)) = 0, \quad \beta \in \mathcal{B},
\end{equation}
has dimension $m - \mbox{rank}\, \Omega(x)$.
In the particular case $\det \Omega(x) \ne 0$, equation~\eqref{homo} has only the trivial solution.
\end{thm}

\begin{proof}
Let $\beta \in \mathcal{B}$ be a solution of \eqref{homo}. Then by \eqref{fromcont4},
$$
 	\beta (I + \tilde R(x)) (I - R(x)) = \beta - \beta \tilde \Theta(x) = \beta - K(x) \psi(x) = 0,
$$
where $\psi(x) = \vv(x, \rho)_{|\mathcal{B}}$ and $K(x)$ is a row vector, which can be found from \eqref{defTheta}.
Using \eqref{eqpsi}, we get
$$
 	\beta (I + \tilde R(x)) = K(x) \psi(x) (I + \tilde R(x)) = K(x) \Omega(x) \tilde \psi(x) = 0.
$$
Let $\mathcal{K}(x)$ be the space of solutions of the equation $K(x) \Omega(x) = 0$ for each fixed $x \in [0, \pi]$.
It is clear that $\dim \mathcal{K}(x) = m - \mbox{rank} \, \Omega(x)$, and $\beta$ is
a solution of \eqref{homo} if and only if $\beta = K(x) \psi(x)$, $K(x) \in \mathcal{K}(x)$,
this yields the assertion of the Theorem.
\end{proof}

\begin{thm} \label{thm:main}
For each fixed $x \in [0, \pi]$, such that $\det \Omega(x) \ne 0$,
equation
\begin{equation} \label{main}
z(x) (I + \tilde R(x)) = \tilde \psi(x),
\end{equation}
has a unique solution $z(x) = (\Omega(x))^{-1} \psi(x)$ in the Banach space $\mathcal{B}$. 	
\end{thm}

Equation \eqref{main} is called {\it the main equation} of Inverse Problem 1.

\begin{remark}
For the matrix Sturm-Liouville operator (when $Q_1(x) \equiv 0$),
we have $\Omega(x) = I$, $\Lambda(x) = 0$. Therefore, the main equation
is uniquely solvable for all $x$ in $[0, \pi]$ and
by \eqref{fromcont4} the operators $(I + \tilde R(x))$ and $(I + R(x))$ are inverses of each other.
\end{remark}

\begin{remark} \label{rem:Fredholm}
The operators $R(x)$ and $\tilde R(x)$ can be approximated by sequences of finite-dimensional operators
in $\mathcal{B}$. Indeed, let
$R^s_{k,n} = R_{k, n}$ for all $n \in \mathbb{Z}$, $1 \le |k| \le s$, and all the other components of
$R^s$ be equal to zero. Using the estimates for $R_{k,n}$ from the proof of Lemma~\ref{lem:Rbound},
it is easy the show that
$\lim\limits_{s \to \infty} \| R^s - R \|_{\mathcal{B} \to \mathcal{B}} = 0$.

Therefore the operator $\tilde R(x)$ is compact,
and one can apply the first Fredholm theorem to the operator $(I + \tilde R(x))$ and obtain from
Theorem~\ref{thm:homo}
even more general result than Theorem~\ref{thm:main}:
if $\det \Omega(x) \ne 0$,
the operator $(I + \tilde R(x))$
has a bounded inverse on $\mathcal{B}$.
\end{remark}

\bigskip

{\large \bf 5. Algorithm for the solution}

\bigskip

In this section, using the main equation \eqref{main} and Theorem~\ref{thm:main},
we provide a constructive algorithm for the solution of Inverse Problem~1.

Let $x \in [0, \pi]$ be a fixed point, such that $\det \Omega(x) \ne 0$.
Suppose we have solved equation \eqref{main} and found $z(x)$. Multiplying \eqref{cont1}
by $(\Omega(x))^{-1}$, we get a formula to construct $z(x, \rho) := (\Omega(x))^{-1} \vv(x, \rho)$ from
$z(x) = z(x, \rho)_{|\mathcal{B}}$.

In a similar way to Lemma~\ref{lem:contour}, it can be proved that
\begin{equation*}
 	\Omega(x) \tilde \Phi(x, \rho) =  \Phi(x, \rho) + 
 	\sum_{n \in w} \sum_{q = 1}^m \Bigl\{ \vv(x, \rho_{nq}) m_{nq}^{-1} \al_{nq} \tilde E(x, \rho, \rho_{nq}) -
 	\vv(x, \tilde \rho_{nq}) \tilde m_{nq}^{-1} \tilde \al_{nq} \tilde E(x, \rho, \tilde \rho_{nq}) \Bigr\},
\end{equation*}
$$
 	\tilde E(x, \rho, \theta) := \frac{\langle \tilde \vv^*(x, \theta), \tilde \Phi(x, \rho) \rangle}{\rho - \theta}.
$$
Therefore, knowing $z(x, \rho)$, we can build the matrix function
$w(x, \rho) := (\Omega(x))^{-1} \Phi(x, \rho)$ by the formula
\begin{equation} \label{defw}
 	w(x, \rho) = \tilde \Phi(x, \rho) -  \sum_{n \in w} \sum_{q = 1}^m \Bigl\{ z(x, \rho_{nq}) m_{nq}^{-1} \al_{nq} \tilde E(x, \rho, \rho_{nq}) -
 	z(x, \tilde \rho_{nq}) \tilde m_{nq}^{-1} \tilde \al_{nq} \tilde E(x, \rho, \tilde \rho_{nq}) \Bigr\}.
\end{equation}

Denote $z^*(x, \rho) := z^{\dagger}(x, \bar \rho)$, $w^*(x, \rho) := w^{\dagger}(x, \bar \rho)$.
Then $\vv^*(x, \rho) = \vv^{\dagger}(x, \bar \rho) = z^*(x, \rho) \Omega^{\dagger}(x)$,
$\Phi^*(x, \rho) = \Phi^{\dagger}(x, \bar \rho) = w^*(x, \rho) \Omega^{\dagger}(x)$.
Using \eqref{inverseform}, we get
$$
 	0 = \vv \Phi^* - \Phi \vv^* = \Omega (z w^* - w z^*) \Omega^{\dagger}.
$$
$$
 	I = \vv {\Phi^*}' - \Phi {\vv^*}' = \Omega (z w^* - w z^*) {\Omega^\dagger}' + \Omega (z {w^*}' - w {z^*}') \Omega^{\dagger}.
$$
Therefore,
$$
 	z w^* - w z^* = 0, 		
$$
\begin{equation} \label{Omega2}
 	(z {w^*}' - w {z^*}')^{-1} = \Omega^{\dagger} \Omega. 	
\end{equation}

It follows from \eqref{wronphi} that
$\langle z^* \Omega^{\dagger}, \Omega z \rangle = 0.$
Consequently,
\begin{equation} \label{Omega3}
  	z^* ({\Omega^{\dagger}}' \Omega - \Omega^{\dagger} \Omega') z = z^* \Omega^{\dagger} \Omega z' - {z^*}' \Omega^{\dagger} \Omega z.
\end{equation}
If the functions $z$, $z^*$, $w$ and $w^*$ are known, one can find $\Omega^{\dagger} \Omega$ from \eqref{Omega2}.
Then the right-hand side of \eqref{Omega3} is known, and we can calculate ${\Omega^{\dagger}}' \Omega - \Omega^{\dagger} \Omega'$.
Differentiating $\Omega^{\dagger} \Omega$, we find ${\Omega^{\dagger}}' \Omega + \Omega^{\dagger} \Omega'$. Then we can find
$\Omega^{\dagger} \Omega'$ and calculate
$$
	a(x) := \Omega^{-1} \Omega' = (\Omega^{\dagger} \Omega)^{-1} \Omega^{\dagger} \Omega'.
$$
Solving the Cauchy problem
$$
 	 \Omega'(x) = a(x) \Omega(x), \quad \Omega(0) = I,
$$
we get $\Omega(x)$ and then can find $\vv(x, \rho)$.

We arrive at the following algorithm for the solution of Inverse Problem 1.

\medskip

{\bf Algorithm 2.} Let the spectral data of the pencil $L$ be given.

\begin{enumerate}
\item Choose a model pencil $\tilde L$, using Algorithm 1.
\item Construct the space $\mathcal{B}$, associated with the spectral data of $L$ and $\tilde L$,
$\tilde \psi(x)$ and $\tilde R(x)$.
\item Find $z(x)$, solving the main equation \eqref{main}.
\item Construct $z(x, \rho)$ and $w(x, \rho)$, using \eqref{cont1} and \eqref{defw}, then get $\Omega^{\dagger}(x) \Omega(x)$
via \eqref{Omega2}.
\item Find $\Omega(x)$ and $\vv(x, \rho) = \Omega(x) z(x, \rho)$.	
\item Substituting $\vv(x, \rho)$ into \eqref{eqv} and \eqref{BC}, obtain
$Q_s(x)$, $h_s$, $H_s$, $s = 1, 2$.
\end{enumerate}

Note that Algorithm 2 works only for the case $\det \Omega(x) \ne 0$
for all $x \in [0, \pi]$.
Further we provide a modification of this algorithm for the general case.

First let us modify step 4 of Algorithm 1 to construct a pencil $\tilde L$
with spectral data satisfying \eqref{asymptrho} and \eqref{asymptalpha}
(with the same constants as the spectral data of $L$) and moreover, satisfying the condition
\begin{equation} \label{Qde}
 	\tilde Q_1(x) = Q_1(x), \quad 0 < x \le \de < \pi.
\end{equation}
Note that the equation
\begin{equation} \label{eqvH1}
 	(I + \tilde H_1)^{-1} (I - \tilde H_1) = \mathcal{O},
\end{equation}
where the matrix $\mathcal{O}$ is unitary, has a solution
\begin{equation} \label{soleqvH1}
 	\tilde H_1 = (I + \mathcal{O})^{-1} (I - \mathcal{O}),
\end{equation}
if and only if $\det(I + \mathcal{O}) \ne 0$. In this case, the matrix $\tilde H_1$ satisfies (I) and (II), i.e.
$\det(I \pm \tilde H_1) \ne 0$ and $\tilde H_1 = -\tilde H_1^{\dagger}$.

\medskip

{\bf Algorithm 3.} Let $SD$ of $L$ and the potential $Q_1(x)$ for $0 < x \le \de < \pi$ be given.

\begin{enumerate}
\item Construct $\tilde h_1$ and $\tilde A$ by steps 1--3 of Algorithm 1.
\item Look for $\tilde Q_1(x)$ in the form
$$
 	\tilde Q_1(x) = \left\{
 	\begin{array}{ll}
 		Q_1(x),  & x \le \de, \\
 		Q_1(\de) + i \kappa (x - \de) I & x \ge \de,
 	\end{array}
 	\right.
 	\quad \kappa \in \mathbb{R}.
$$
Then $\tilde P_{\pm}(x) = \exp\left\{ \pm \bigl(Q_1(\de) (x - \de) + \frac{i \kappa}{2} (x - \de)^2 I\bigr) \right\} P_{\pm}(\de)$,
and $\tilde H_1$ must satisfy the relation
$$
 	(I + \tilde H_1)^{-1}(I - \tilde H_1) = \tilde P_{-}(\pi) \tilde A (I - \tilde h_1)(I + \tilde h_1)^{-1} \tilde P_{+}^{-1}(\pi).
$$
It is equation \eqref{eqvH1} with $\mathcal{O}$ in the form
$$
 	\mathcal{O}(\kappa) = \exp(-i\kappa(\pi-\de)^2) \mathcal{U}, \quad \mathcal{U}^\dagger = \mathcal{U}^{-1}.
$$
Take such $\kappa$, that $\det (\mathcal{O}(\kappa) + I) \ne 0$.
\item Construct $\tilde H_1$ via \eqref{soleqvH1}.
\item Put $\tilde Q_0(x) = \tilde h_0 = \tilde H_0 = 0$, $x \in [0, \pi]$.
\end{enumerate}

\medskip

Note that if one chooses $\kappa$ bounded (say, $|\kappa| < 1$), then $\| \tilde Q_1(x) \| \le C$, where $C$
depends only on $Q_1$ and does not depend on $\de$.

\medskip

{\bf Algorithm 4.} {\it Solution of Inverse Problem 1.}
Let the spectral data of the pencil $L$ be given.

\begin{enumerate}
\item Put $k := 1$, $\de_0 := 0$.
\item We assume that for $x \in [0, \de_{k - 1}]$ the potential $Q_1(x)$
is already known. Choose the model pencil $\tilde L$ satisfying \eqref{Qde},
using Algorithm 3.
\item Step 2 of Algorithm 2.
\item Choose the largest $\de_k$, $\de_{k - 1} < \de_k \le \pi$, such that the main equation \eqref{main} has a unique solution
for $x \in (\de_{k-1}, \de_{k})$, find this solution.
\item Steps 4--5 of Algorithm 2 for $x \in (\de_{k - 1}, \de_k)$.
\item Substituting $\vv(x, \rho)$ into \eqref{eqv} and \eqref{BC}, obtain
$Q_s(x)$, $x \in (\de_{k - 1}, \de_k)$, $h_s$, $s = 1, 2$. If $\de_k = \pi$, construct also $H_s$, $s = 1,2$.
\item If $\de_k = \pi$, then terminate the algorithm, otherwise put $k := k + 1$ and go to step 2.
\end{enumerate}

\begin{thm}
Algorithm 4 terminates after a finite number of steps.
\end{thm}

\begin{proof}
Consider the pencil $\tilde L$ constructed by Algorithm 3.

We have $P_{\pm}(x) = \tilde P_{\pm}(x)$ for $x \in [0, \de]$.
%Using Lemma~\ref{lem:P} and \eqref{defOmega}, we get $\Omega(x) = I$, $x \in [0, \de]$.
For $x > \de$, the following integral representations are valid
$$
 	P_{\pm}(x) = P_{\pm}(\de) \pm \int_{\de}^{x} Q_1(t) P_{\pm}(t) \, dt, \quad
 	\tilde P_{\pm}(x) = P_{\pm}(\de) \pm \int_{\de}^x \tilde Q_1(x) \tilde P_{\pm}(t) \, dt.
$$
Then $\| P_{\pm}(x) \|, \| \tilde P_{\pm}(x) \| \le C$, where the constant $C$ depends only on $Q_1$.
Using \eqref{defOmega}, we obtain
$$
 	\Omega(x) = I + B_{\de}(x), \quad \| B_{\de}(x) \| \le C_{Q_1}(x - \de).
$$

Choose an integer $s$ such that $\pi / s < 1 / (2 C_{Q_1})$.
Put $x_k := k \pi / s$, $k = \overline{0, s}$.
Then if we construct $\tilde Q_1$ by Algorithm 3 with $\de = x_{k - 1}$,
we get $\| B_{\de}(x) \| < 1/2$, $x \in (x_{k-1}, x_k)$. Hence $\det \Omega(x) \ne 0$ for $x \in (x_{k - 1}, x_k)$
and the main equation \eqref{main} is uniquely solvable in this interval.

It is easy to show that $\de_k \ge x_k$ for all $k = \overline{1, s}$. Since $x_s = \pi$,
the number of steps in Algortihm~4 is finite.
\end{proof}

\bigskip

{\large \bf 6. Multiple poles of the Weyl matrix}

\bigskip

In this section, we consider the general case, when $N$ is not necessary equal zero and the Weyl matrix 
can have multiple poles.
For convenience, we assume that for $|n| \ge N$ the assertion of Theorem~\ref{thm:SD} holds for 
both collections $\{ \rho_{nq}, \al_{nq}\}$ and $\{ \tilde \rho_{nq}, \tilde \al_{nq}\}$.

Let a contour $\mathcal{C}_0$ enclose the set $\{ \rho_{nq}, \, \tilde \rho_{nq} \}_{|n| < N}$ and do not contain
other eigenvalues $\rho_{nq}$, $\tilde \rho_{nq}$, $|n| \ge N$, inside or on the boundary.
Then instead of relations \eqref{cont3} and \eqref{cont4}, we get
\begin{multline} \label{cont3int}
 	\tilde \Omega(x) \vv(x, \rho) =  \tilde \vv(x, \rho) - \frac{1}{2 \pi i} \int_{\mathcal{C}_0}
 	\tilde \vv(x, \theta) \left(M(\theta) - \tilde M(\theta)\right) D(x, \rho, \theta) \, d \theta \\ -
 	\sum_{|n| \ge N} \sum_{q = 1}^m \Bigl\{ \tilde \vv(x, \rho_{nq}) m_{nq}^{-1} \al_{nq} D(x, \rho, \rho_{nq}) -
 	\tilde \vv(x, \tilde \rho_{nq}) \tilde m_{nq}^{-1} \tilde \al_{nq} D(x, \rho, \tilde \rho_{nq}) \Bigr\},
\end{multline}
\begin{multline} \label{cont4int}
 	D(x, \rho, \theta) - \tilde D(x, \rho, \theta)
 	+ \tilde \vv^*(x, \theta) \tilde \Lambda(x) \vv(x, \rho) + \frac{1}{2 \pi i} \int_{\mathcal{C}_0}
 	\tilde D(x, \xi, \theta) \left(M(\xi) - \tilde M(\xi) \right)
 	D(x, \rho, \xi) \, d \xi \\ +
 	\sum_{|n| \ge N} \sum_{q = 1}^m \Bigl\{
 	 \tilde D(x, \rho_{nq}, \theta) m_{nq}^{-1} \al_{nq} D(x, \rho, \rho_{nq}) -
 	 \tilde D(x, \tilde \rho_{nq}, \theta) \tilde m_{nq}^{-1} \tilde \al_{nq} D(x, \rho, \tilde \rho_{nq})
 	\Bigr\} = 0.
\end{multline}

Substituting \eqref{reprM} into int integral in \eqref{cont3int}, we obtain
\begin{equation*}
\frac{1}{2 \pi i} \int_{\mathcal{C}_0}
 	\tilde \vv(x, \theta) \left(M(\theta) - \tilde M(\theta)\right) D(x, \rho, \theta) \, d \theta =
 	\sum_{n = 1}^K \tilde \vv_{n, 0}(x) A_{n, 0}(x, \rho) -
 	\sum_{n = 1}^{\tilde K} \tilde \vv_{n, 1}(x) A_{n, 1}(x, \rho).
\end{equation*}
where
$$
	\vv_{n + k, 0}(x) := \frac{1}{k!}\frac{\partial^{k}}{\partial \rho^{k}} \vv(x, \rho_n), \quad
 	A_{n + k, 0}(x, \rho) = \sum_{\nu = k}^{m_n - 1} \frac{1}{(\nu - k)!} \al_{n + \nu} \frac{\partial^{\nu - k}}{\partial \theta^{\nu - k}} D(x, \rho, \rho_n),
$$
$$
	n \in S, \, k = \overline{0, m_n -1},
$$
$$
	\vv_{n + k, 1}(x) := \frac{1}{k!}\frac{\partial^{k}}{\partial \rho^{k}} \vv(x, \tilde \rho_n), \quad
 	A_{n + k, 1}(x, \rho) = \sum_{\nu = k}^{\tilde m_n - 1} \frac{1}{(\nu - k)!} \tilde \al_{n + \nu} \frac{\partial^{\nu - k}}{\partial \theta^{\nu - k}} D(x, \rho, \tilde \rho_n),
$$
$$
	n \in \tilde S, \, k = \overline{0, \tilde m_n -1}.
$$

It is easy to check, that
$$
 	\lim_{\substack{R \to \infty \\ \Gamma_R \subset G_{\de} }} \frac{\partial^{\nu}}{\partial \rho^{\nu}}\eps_R(x, \rho) = 0, \quad
 	\lim_{\substack{R \to \infty \\ \Gamma_R \subset G_{\de} }} \frac{\partial^{\nu + s}}{\partial \rho^{\nu} \partial \theta^s}\eps_{R, jk}(x, \rho, \theta) = 0, \quad
 	j, k = 1, 2,
$$
uniformly with respect to $x \in [0, \pi]$ and $\rho, \, \theta$ on bounded sets,
for $\eps_R(x, \rho)$ and $\eps_{R, jk}(x, \rho, \theta)$ from the proof of Lemma~\ref{lem:contour}.
Hence one can differentiate \eqref{cont3int} by $\rho$ and
\eqref{cont4int} by $\rho$ and $\theta$. From \eqref{cont3int} we get
\begin{multline} \label{cont3dif}
 	\tilde \Omega(x) \frac{\partial^{\nu}}{\partial \rho^{\nu}}\vv(x, \rho) =
 	\frac{\partial^{\nu}}{\partial \rho^{\nu}}\tilde \vv(x, \rho) -
 	\left( \sum_{n = 1}^K \tilde \vv_{n, 0}(x) \frac{\partial^{\nu}}{\partial \rho^{\nu}} A_{n, 0}(x, \rho) -
 	\sum_{n = 1}^{\tilde K} \tilde \vv_{n, 1}(x)
 	\frac{\partial^{\nu}}{\partial \rho^{\nu}} A_{n, 1}(x, \rho) \right)
 	\\-	\sum_{|n| \ge N} \sum_{q = 1}^m \left\{ \tilde \vv(x, \rho_{nq}) m_{nq}^{-1} \al_{nq} \frac{\partial^{\nu}}{\partial \rho^{\nu}} D(x, \rho, \rho_{nq}) -
 	\tilde \vv(x, \tilde \rho_{nq}) \tilde m_{nq}^{-1} \tilde \al_{nq} \frac{\partial^{\nu}}{\partial \rho^{\nu}} D(x, \rho, \tilde \rho_{nq}) \right\}.
\end{multline}
Relations \eqref{cont1} and \eqref{cont2} can be modified in a similar way.

Consider the finite-dimensional space
$$
	\mathcal{B}_0 := \bigl\{ b = \{b_k\}_{k = 1}^{K + \tilde K} \colon b_k \in \mathbb{C}^{m \times m} \bigr\},
	\quad
	\| b \|_{\mathcal{B}_0} = \max_{1 \le k \le K + \tilde K} \| b_k \|.
$$
Denote
$$
 	\vv(x, \rho)_{| \mathcal{B}_0} := \left\{ \vv_{1,0}(x), \dots, \vv_{K,0}(x), \vv_{1,1}(x), \dots, \vv_{\tilde K, 1}(x)\right\} \in \mathcal{B}_{0}.
$$
Let $G_k$, $k \in \mathbb{Z} \backslash \{ 0 \}$ be the sets $G_{ns}$, $|n| \ge N$, $q = \overline{1, p}$, renumerated in a natural way.
Instead of $\mathcal{B}$, introduce the following Banach space
$$
	\mathcal{B}_* =\mathcal{B}_{0}\times\mathcal{B}_1 = \bigl\{
	f=\{f_k\}_{k\in\mathbb{Z}},\, f_0\in\mathcal{B}_{0},\, f_k\in\mathcal{B}(G_k),\, k\in\mathbb{Z}\backslash\{0\} \colon
	\| f \|_{B_*} := \sup_{k \in \mathbb{Z}} \| f_k \| < \infty \bigr\}.
$$

Fix $x \in [0, \pi]$. Then $\vv(x, \rho)$ produces an element of $\mathcal{B}_*$ by the following rule
$$
 	\psi(x) := \vv(x, \rho)_{|\mathcal{B}_*} = \{ \psi_k(x) \}_{k \in \mathcal{Z}}, \quad
 	\psi_0(x) = \vv(x, \rho)_{| \mathcal{B}_{0}}, \quad \psi_k(x) = \vv(x, \rho)_{\mathcal{B}(G_k)}, \quad k \in \mathbb{Z} \backslash \{ 0 \}.
$$
Similarly $\tilde \psi(x) := \tilde \vv(x, \rho)_{|\mathcal{B}_{*}}$.
Then the \eqref{cont3dif} and the analogue of \eqref{cont1}
can be transformed into the following equations in the Banach space $\mathcal{B_*}$:
$$
 	\Omega(x) \tilde \psi(x) = \psi(x) (I + \tilde R(x)), \quad
 	\tilde \Omega(x) \psi(x) = \tilde \psi(x) (I - R(x)),
$$
where $I$ is the identity operator in $\mathcal{B_*}$, and $R(x)$, $\tilde R(x)$ are linear operators, acting from
$\mathcal{B_*}$ to $\mathcal{B_*}$.

According to \eqref{cont3dif}, the operator $R(x)$ acts to an arbitrary element $\psi \in \mathcal{B}$ in the following way:
\begin{equation} \label{Rseries}
 	\psi R(x) = \sum_{k \in \mathbb{Z}} \psi_k R_k(x), \quad \psi = \{ \psi_k\}_{k \in \mathbb{Z}},
\end{equation}
$$
 	\psi_0 R_0(x) = \left[ \sum_{j = 1}^K \psi_{0, j} A_{j, 0}(x, \rho) -
 	\sum_{j = 1}^{\tilde K} \psi_{0, K + j} A_{j, 1}(x, \rho) \right]_{|\mathcal{B}_*},	
$$
$$
 	\psi_k R_k(x) = \left[ \sum_{n, q \colon \rho_{nq} \in G_k} \psi_k(\rho_{nq}) m_{nq}^{-1} \al_{nq} D(x, \rho, \rho_{nq}) -
 	\sum_{n, q \colon \tilde \rho_{nq} \in G_k} \psi_k(\tilde \rho_{nq}) \tilde m_{nq}^{-1} \tilde \al_{nq} D(x, \rho, \tilde \rho_{nq})\right]_{|\mathcal{B}_*}.
$$

The results of Lemma~\ref{lem:Rbound}, Theorems~\ref{thm:homo} and~\ref{thm:main} and Remark~\ref{rem:Fredholm}
are valid for this operator as well as for one considered in Section 4.
Namely, for each fixed $x \in [0, \pi]$,
\begin{itemize}
\item The series in \eqref{Rseries} converges absolutely and the operator $R(x)$ is bounded in $\mathcal{B_*}$.
The proof is analogous to Lemma~\ref{lem:Rbound}, since the infinite-dimensional part of the considered operators is the same
(given by $R_{k, n}(x))$.

\item Equation \eqref{main} (considered now as an equation in $\mathcal{B}_*$) has a unique solution
$z(x) \in \mathcal{B}_*$ for such $x \in [0, \pi]$ that $\det \Omega(x) \ne 0$.
To prove this, obtain from \eqref{cont4int} a relation in $\mathcal{B}_*$,
similar to \eqref{fromcont4}. Moreover, the operator $\tilde R(x)$
is compact in $B_*$ and the operator $(I + \tilde R(x))$ is invertible on $\mathcal{B}_*$.

\item Equation \eqref{main} in the Banach space $\mathcal{B}_{*}$ can be used in Algorithm 4
(with obvious modifications of the Algorithm)
for the recovery of the pencil $L$ from its spectral data $SD$.
\end{itemize}

{\bf Acknowledgement.} This research was supported by Grant 1.1436.2014K
of the Ministry of Education and Science of the Russian Federation and by Grants 13-01-00134, 14-01-31042 and 15-01-04864
of Russian Foundation for Basic Research.

\vspace{1cm}

Natalia Bondarenko

Department of Mechanics and Mathematics

Saratov State University

Astrakhanskaya 83, Saratov 410012, Russia

{\it bondarenkonp@info.sgu.ru}

\end{document}